\documentclass{article}
\usepackage{amssymb, amsmath,amsrefs,graphicx}
\usepackage{color}
\usepackage{float}

\def\dsp{\displaystyle}

\newcommand{\Pm}{\vect P}
\newcommand{\Mm}{\vect M}
\newcommand{\Bm}{B}
\def\Ac{{\cal A}}
\newcommand{\tAc}{{{\cal A}}}
\newcommand{\rmi}{{{ i }}}

\def\vect#1{\mbox{\boldmath{$#1$}}}

\def\vx{\vec{\vect x}}
\def\vy{\vec{\vect y}}
\def\vz{{ \vec{\vect z}}}

\newcommand{\ta}{{a}}

\newcommand{\smT}{{\gamma}}

\def\mC{\mathbb{C}}

\def\wG{G}

\def\wf{f}
\def\wl{{ l}}
\def\we{e}

\def\wg{g}
\def\wh{{h}}

\newcommand{\Rho}{\cal{X}}

\newcommand{\eps}{\varepsilon}
\renewcommand{\epsilon}{\eps}
\renewcommand{\leq}{\leqslant}
\renewcommand{\geq}{\geqslant}

\def\bfrho{\mbox{\boldmath$\rho$}}
\def\bfb{\mbox{\boldmath$b$}}

\def\bfm{\mbox{\boldmath$m$}}

\newcommand{\supp}{\mathrm{supp}}

\newtheorem{theorem}{Theorem}
\newtheorem{lemma}{Lemma}
\newtheorem{remark}{Remark}
\newtheorem{proposition}{Proposition}
\newtheorem{definition}{Definition}
\newenvironment{proof}{\paragraph{{\bf Proof:}}}{\hfill$\square$}

\title{Robust multifrequency imaging with MUSIC}

\author{Miguel Moscoso%
\thanks{Department of Mathematics, Universidad Carlos III de Madrid, Leganes, Madrid 28911, Spain ({moscoso@math.uc3m.es})} 
\and
Alexei Novikov%
\thanks{Mathematics Department, Penn State University, University Park, PA 16802 ({novikov@psu.edu})}
\and
George Papanicolaou%
\thanks{Department of Mathematics, Stanford University, Stanford, CA 94305 ({papanicolaou@stanford.edu})}
\and
Chrysoula Tsogka%
  \thanks{Applied Math Unit, University of California, Merced, 5200 North Lake Road, Merced, CA 95343 ({ctsogka@ucmerced.edu}).}%
 }

\begin{document}

\maketitle
\begin{abstract}
In this paper, we study the MUltiple SIgnal Classification (MUSIC) algorithm often used to image small targets when multiple measurement vectors are available. We show that this algorithm may be used when the imaging problem can be cast as a linear system that admits a special factorization. We discuss several active array imaging configurations where this factorization is exact, as well as other configurations where the factorization only holds approximately and, hence, the results provided by MUSIC deteriorate. We give special attention to the most general setting where an active array with an arbitrary number of transmitters and receivers uses signals of multiple frequencies  to image the targets. 
This setting provides all the possible diversity of information that can be obtained from the illuminations. We  give a theorem that shows that MUSIC is robust with respect to additive noise provided that the targets are well separated. The theorem also shows the relevance of using appropriate sets of controlled parameters, such as excitations, to form the images with MUSIC robustly.  We present numerical experiments that support our theoretical results.
\end{abstract}

\vspace{2pc}


\section{Introduction}
Imaging is an inverse problem in which we seek to reconstruct a medium's characteristics, such as the reflectivity,
by recording its response to one or more known excitations. The output is usually an image 
giving an estimate of an unknown characteristic in a bounded domain, the imaging window of interest. Although this 
problem is in all generality non-linear, it is often adequately formulated as a set of $\aleph$ linear systems of the form
\begin{equation}\label{family1intro}
\Ac_{\wl_q} \bfrho=\bfb_{\wl_q}\,, \quad q=1,\dots,\aleph. 
\end{equation}
Here, $ \bfrho\in\mC^K$ is the unknown vector we seek to estimate and 
$\bfb_{\wl_q}\in\mC^N$ are different measurement vectors.
The  essential  point  in (\ref{family1intro}) is that the model matrix $\Ac_{\wl_q}$
depends on a  parameter vector $\vect \wl_q=[l_{1q},l_{2q},\dots,l_{Kq}]^{\intercal}$  that contains the
experimental constants $l_{jq}$, such as the excitations, that we control and change  to  form  the  images.
To simplify the notation, we will denote the different excitations
by the scalar $q$ and write $\Ac_{q} \bfrho=\bfb_{q}$ instead, unless it is necessary to explicitly state that the model matrix, 
and the measurements, depend on a vector $\vect \wl_q$. 
We are interested in underdetermined linear systems, so $N<K$, where the unknown vector is M-sparse with $M \ll K$. 

To solve (\ref{family1intro}) we consider the 
MUltiple SIgnal Classification (MUSIC) algorithm which has been used successfully in signal processing 
\cite{Schmidt86,Kaveh86,Hayes96,Liao15,Liao16} and imaging \cite{DeGraaf98,DMG05, Rubsamen09, Borcea07,Fannjiang11,Griesmaier17}. 
In this work  we make the fundamental observation 
that the MUSIC algorithm gives the exact support of the solution 
of (\ref{family1intro}), in the noise free case, when the matrices $\Ac_{q}$ admit the following factorization
\begin{equation}\label{factor}
 \Ac_{q} =  \tAc \;  \Lambda_{q}, \,\,{\rm with }\ \Lambda_{q} \,\, {\rm diagonal},
 \end{equation}
and $\tAc$ independent of the parameter vector $\vect \wl_q$.
In this case, (\ref{family1intro}) can also be formulated as the Multiple Measurement Vector (MMV) problem
\begin{equation}\label{eq:mmv}
 \tAc  \, \bfrho_q =\bfb_{q}, \,\,{\rm with }\  \ \bfrho_q=   \Lambda_{q}\, \bfrho.
\end{equation}
Here,  the multiple unknown vectors $\bfrho_q$ share the same support $T=\supp(\bfrho)$, with $|T|=M$. The MMV formulation
is usually written as a matrix-matrix equation
\begin{equation}\label{eq:mmv_matrix}
 \tAc  \, \Rho = B\, , 
\end{equation}
where the unknown is now the matrix $\Rho \in  \mC^{K \times \aleph}$  
whose columns are the vectors $\bfrho_q=\Lambda_{q} \bfrho$, and $\Bm \in \mC^{N \times \aleph}$ is the data or observation matrix whose columns are the vectors $\bfb_q$. 

The main advantage of the MMV formulation  is 
that we can immediately infer that the data vectors $\bfb_q$ are linear combinations of the same M-columns of $\tAc$, those that belong to  $T$. 
The implication is that, in the absence of noise,  the columns of $\tAc$ indexed by $T$  span $R(B)$, the  range or column subspace of $\Bm$.  
Thus, MUSIC finds the support $T$ as the zero set of the orthogonal projections of the columns of $\tAc$ 
onto the left nullspace of the matrix $\Bm$, which is the orthogonal complement of  $R(B)$  and can be easily found with an SVD. Moreover, the support can be recovered exactly with MUSIC
under the assumption that all (M+1)-sets of columns of $\tAc$ are linearly independent.  The support $T$ can be recovered approximately if the data is noisy.
In Theorem~\ref{th:musicnoise} we quantify an acceptable level of noise for such approximate recovery.

The MMV problem can also be solved using an optimization perspective as described in \cite{Cotter05, MCW05, TROPP06-2-0, TROPP06-2}. The main idea is to seek the solution matrix $\Rho$
with the minimal $(2,1)$-norm, which consists in minimizing the $\ell_1$ norm of the vector formed by the $\ell_2$ norms of the rows of the unknown matrix $\Rho$. This guarantees the common support of the 
solution's columns. We do not pursue this approach here and refer the reader to \cite{CMP14} for an application of this formalism to imaging strong scattering scenes as well as to \cite{BMPT16} where an MMV 
formulation for synthetic aperture imaging of frequency and direction dependent reflectivity was introduced and analyzed.

In this paper, we present several configurations in array imaging that can be cast under the general framework discussed here, such as single- and multiple-frequency array imaging using single- or multiple-receivers. All 
these problems can be formulated as (\ref{family1intro}) in which multiple measurement vectors are recorded. We show that some array imaging problems  admit the factorization (\ref{factor}) and, thus, the support of the 
unknown can be  recovered exactly by MUSIC. However, there are other configurations such as multiple frequency imaging with several transmitters and receivers for which this factorization is not feasible. Still, we
show that  factorization (\ref{factor}) approximately holds under the paraxial approximation, {\em i.e.},  when the image region is far from the array and is small. 

We also consider the non-linear phase retrieval problem, which according to \cite{Novikov14,Moscoso16,Moscoso17} can be reduced to a linear system of the form (\ref{family1intro}). 
This requires intensity data corresponding to multiple coherent illuminations which are transformed to interferometric data using the polarization identity. 
We consider multiple frequency intensity data collected at a single receiver due to multiple coherent illuminations.

To summarize, the main contributions of this work are as follows.  We show (i) in Section  \ref{sec:la} that the support of the solution of (\ref{family1intro}) can be recovered exactly with MUSIC when the (noiseless) data can be structured so that the model matrix admits a factorization in terms of a universal model matrix multiplied by a diagonal matrix that depends on the excitation as in (\ref{factor}). Then the noisy case is considered in Theorem \ref{th:musicnoise} that gives conditions under which MUSIC is robust with respect to additive noise.  We also show (ii) that when we have full data diversity, that is, we have data from multiple sources, multiple receivers and multiple frequencies, then there is a data structure that is associated with a model matrix that admits an approximate factorization (\ref{factor}) in particular imaging regimes such as the paraxial regime that is considered in Section \ref{sec:ds}. As a consequence, MUSIC can be used with full interaction over multiple frequencies to image in this regime as illustrated in Section \ref{sec:numerics}.

 The paper is organized as follows.  In Section \ref{sec:array} we present the active array 
 imaging problem and its linear algebra formulation.  In Section \ref{sec:la} we discuss in an abstract linear algebra framework the conditions under 
 which MUSIC provides the exact solution to the MMV problem (\ref{eq:mmv}) and analyze its performance for noisy data. In Section \ref{sec:ds} 
 we consider some common configurations used in active array imaging and discuss the adequate data-structures to be 
 used in imaging with MUSIC. In particular, Section \ref{sec:ds} contains a description of our approximate MUSIC for multiple frequency imaging with several transmitters and receivers.
  In Section \ref{sec:numerics}, we explore with numerical simulations the performance of multifrequency MUSIC
 with intensity-only data. Section \ref{sec:concl} contains our conclusions.

\section{The active array imaging problem}
 \label{sec:array}
 The goal of array imaging is to form images inside a region of interest called the image window IW. 
In active array imaging the array probes the medium by sending signals and recording the echoes. 
Probing of the medium can be done with many different types of arrays that differ in their number of 
transmitters and receivers, their geometric layouts, or the type of signals they use for illumination. 
Moreover, they may use single frequency signals sent from different positions, or multifrequency signals 
sent from one or more positions.  
Obviously, the problem of active array imaging also depends 
on the receivers. They can record the intensities and phases of the signals that arrive to the array or only their intensities.

   \begin{figure}[htbp]
    \begin{center}
 \begin{picture}(0,0)%
\includegraphics[scale=1]{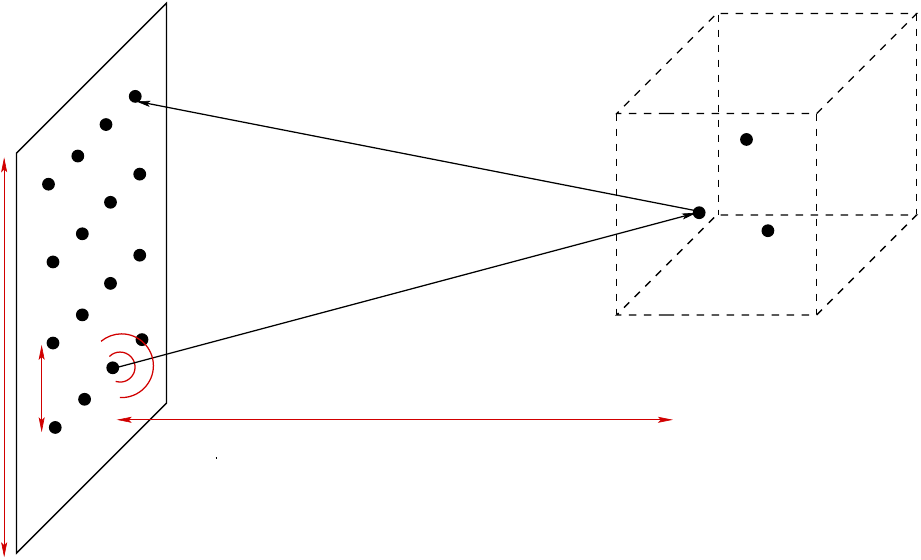}%
\end{picture}%
\setlength{\unitlength}{1579sp}%
\begingroup\makeatletter\ifx\SetFigFont\undefined%
\gdef\SetFigFont#1#2#3#4#5{%
  \reset@font\fontsize{#1}{#2pt}%
  \fontfamily{#3}\fontseries{#4}\fontshape{#5}%
  \selectfont}%
\fi\endgroup%
\begin{picture}(11024,6695)(399,-6908)
\put(9451,-4411){\makebox(0,0)[lb]{\smash{{\SetFigFont{9}{10.8}{\familydefault}{\mddefault}{\updefault}{\color[rgb]{0,0,0}IW}%
}}}}
\put(1651,-1486){\makebox(0,0)[lb]{\smash{{\SetFigFont{9}{10.8}{\familydefault}{\mddefault}{\updefault}{\color[rgb]{0,0,0}${\vx}_r$}%
}}}}
\put(1951,-4111){\makebox(0,0)[lb]{\smash{{\SetFigFont{9}{10.8}{\familydefault}{\mddefault}{\updefault}{\color[rgb]{0,0,0}$\lambda$}%
}}}}
\put(1501,-4786){\makebox(0,0)[lb]{\smash{{\SetFigFont{9}{10.8}{\familydefault}{\mddefault}{\updefault}{\color[rgb]{0,0,0}${\vx}_s$}%
}}}}
\put(4126,-5086){\makebox(0,0)[lb]{\smash{{\SetFigFont{9}{10.8}{\familydefault}{\mddefault}{\updefault}{\color[rgb]{0,0,0}$L$}%
}}}}
\put(526,-4036){\makebox(0,0)[lb]{\smash{{\SetFigFont{9}{10.8}{\familydefault}{\mddefault}{\updefault}{\color[rgb]{0,0,0}$a$}%
}}}}
\put(8701,-3061){\makebox(0,0)[lb]{\smash{{\SetFigFont{9}{10.8}{\familydefault}{\mddefault}{\updefault}{\color[rgb]{0,0,0}${\vz}_j$}%
}}}}
\put(931,-4951){\makebox(0,0)[lb]{\smash{{\SetFigFont{9}{10.8}{\familydefault}{\mddefault}{\updefault}{\color[rgb]{0,0,0}$h$}%
}}}}
\end{picture}%
    \caption{General setup of an array imaging problem. The transducer at $\vx_s$  emits a probing signal and the reflected signals are recorded at $\vx_r$. The scatterers located at ${\vz}_{j}$, $j=1,\dots,M$ 
    are at distance $L$ from the array and inside the image window IW. }
    \label{cs_3d_illustration}
    \end{center}
    \end{figure}
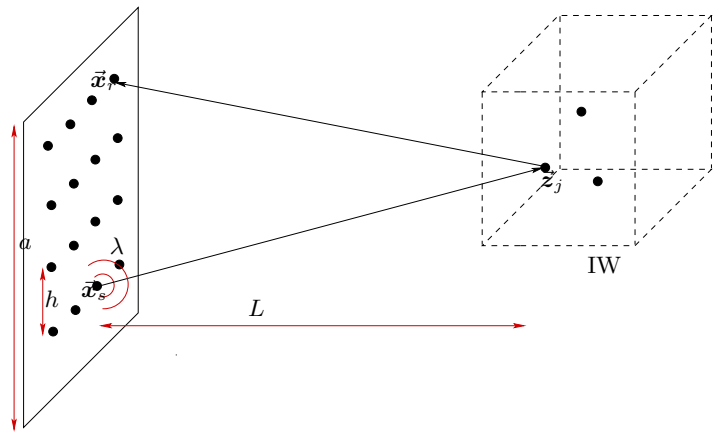

In Figure \ref{cs_3d_illustration}, an array of size $a$ probes the medium by sending and recording signals from positions $\vx_s$ and $\vx_r$,  respectively, $s,r = 1,2,\ldots,N$. It can send signals of one or several frequencies $\omega_l$,  $l=1,\dots,S$.
The goal is to reconstruct a sparse scene consisting of $M$ point-scatterers at a distance $L$ from the array. The positions of the scatterers in the IW are denoted by $\vz_j$, and 
their reflectivities by $\alpha_j \in\mC$, $j=1,\dots,M$. The ambient medium between the array and the scatterers can be homogeneous or inhomogeneous. In this paper, we consider that wave propagation is described by the scalar wave equation. Nevertheless, 
the methodology described here directly extends to other types of vector waves such as electromagnetic waves.

In order to form the images we discretize the IW using a uniform grid of points $\vy_k$, $k=1,\ldots,K$, and we introduce the {\it  true 
reflectivity vector}\footnote{Superscript $\intercal$  here, and throughout the paper, means transpose. It looks similar to $T$ that we use  as the index set of the support of a vector. As such, $T$ appears as a subscript.}

$$\vect \rho=[\rho_{1},\ldots,\rho_{K}]^{\intercal}
\in\mC^K\, ,$$ 
such that
\[
\rho_{k} =
\left\{ 
\begin{array}{ll}
\alpha_j, & \hbox{ if } \| \vz_j - \vy_{k} \|_\infty <   \hbox{grid-size, for some } j=1,\ldots,M,\\ 
0, & \hbox{ otherwise} \\
\end{array}
\right.
\]
We will not assume that the scatterers
lie on the grid, so $\{\vz_1,\ldots,\vz_M \}\not\subset\{\vy_1,\ldots,\vy_K\}$ in general. 
To write the data received on the array in a compact form, we define the Green's function vector   
\begin{equation}\label{eq:GreenFuncVec}
\vect \wg(\vy;\omega)=[\wG(\vx_{1},\vy;\omega), \wG(\vx_{2},\vy;\omega),\ldots,
\wG(\vx_{N},\vy;\omega)]^{\intercal}\,
\end{equation}
at location $\vy$ in the IW, where 
\begin{equation}\label{greenfunc}
\wG(\vx,\vy;\omega)=\frac{\exp(\rmi\kappa|\vx-\vy|)}{4\pi|\vx-\vy|}\, ,\quad
 \kappa=\frac{\omega}{c_0},
\end{equation}
denotes the free-space Green's function of
the background medium. It characterizes the propagation of a 
signal of angular frequency $\omega$ from point $\vy$ to point $\vx$, so 
(\ref{eq:GreenFuncVec}) 
represents the signal received at the array due to a point source of frequency $\omega$
at $\vy$.

We assume that the scatterers are far apart or that the reflectivities are small, so multiple scattering between them is negligible. In this case,  the Born approximation holds and, thus, the response at $\vx_r$ 
due to a pulse of angular frequency $\omega_l$, amplitude one and phase zero sent from $\vx_s$, and reflected by the $M$ scatterers,  is given by 
\begin{equation}
\label{response}
\begin{array}{ll}
P(\vx_r,\vx_s;\omega_l) &\dsp \!\!=\!\! \sum_{j=1}^M\alpha_j G (\vx_r, \vz_j; \omega_l) \, G (\vz_j,\vx_s; \omega_l)  \\
&\dsp \!\!=\!\!  \sum_{k=1}^K\rho_{k}  G (\vx_r, \vy_k; \omega_l)  G (\vy_k,\vx_s; \omega_l).
\end{array}
\end{equation}
When all the sources and the receivers in the array are used for imaging, the data are arranged in the so called single frequency response matrix 
\begin{equation}
\label{responsematrix}
\Pm(\omega_l)=[P(\vx_r,\vx_s;\omega_l)]_{r,s=1}^N
=\sum_{k=1}^K\rho_{k} \vect \wg(\vy_{k}; \omega_l) \, \vect\wg^{\intercal}(\vy_{k}; \omega_l).
\end{equation}
If only one frequency is used to probe the medium, all  the  information available  for  imaging 
is  contained  in  (\ref{responsematrix}). 
The most general configuration is the one of multiple sources, multiple receivers and multiple frequencies. 
In this case, the array response forms a tensor with elements $P(\vx_r,\vx_s;\omega_l)$, $r,s=1,\ldots,N$, and $l=1,\ldots,S$.

\section{The MUSIC algorithm}
\label{sec:la}
MUSIC is a subspace imaging algorithm based on the decomposition of the measurements into two orthogonal domains. The dominant one is due to the
signals and is referred to as the signal subspace, while the other is attributed to the noise and is referred to as the noise subspace. Both are easily found through the SVD of the data matrix
\begin{equation}\label{eq:datamatrix}
 \Bm =
 \left(
 \begin{array}{cccc}
         b_{11} &  b_{12} & \dots & b_{1\aleph}\\
         b_{21} &  b_{22} & \dots & b_{2\aleph} \\
         \dots &  \dots & \dots & \dots\\
          b_{N1} &  b_{N2} & \dots & b_{N\aleph}
        \end{array}
        \right)
 = 
\left(
\begin{array}{cccc}
    \uparrow & \uparrow&        & \uparrow \\
    \vect b_1 & 
    \vect b_2& 
    \ldots & 
    \vect b_\aleph    \\
\downarrow & \downarrow&        & \downarrow
  \end{array}
\right) \,\,\in\mC^{N\times \aleph},
 \end{equation}
whose column vectors $\vect b_q $ 
are obtained from a family of linear systems (\ref{family1intro}).

Our first  result is Proposition \ref{prop1}, which is the key observation that MUSIC provides the exact support of the unknown vector $\bfrho$ when the matrices $\Ac_{q}$ in the original problem (\ref{family1intro}) admit 
a factorization of the form (\ref{factor}). Physically, this factorization means that the data vectors $\vect b_q $ are just different weighted sums of the same columns of the matrix $\tAc$ in (\ref{factor}).

In this framework, we also obtain Theorem \ref{th:musicnoise} which gives conditions for MUSIC to be robust with respect to  noise in the data.

\begin{proposition} \label{prop1}
Assume $\vect \rho \in \mathbb{C}^{K}$ is $M$-sparse with $M<N$, and assume that (\ref{family1intro}) can be rewritten in the form 
\begin{equation}\label{family1}
 \tAc  \, \Lambda_{q} \, \vect \rho = \vect b_q\, , \quad q=1,\dots,\aleph,
\end{equation}
with the matrix
\begin{equation}\label{eq:matrixtilde}
 \tAc 
 = 
\left(
\begin{array}{cccc}
    \uparrow & \uparrow&        & \uparrow \\
    \vect {\ta}_1   & \vect  {\ta}_2& \ldots & \vect  {\ta}_K   \\
\downarrow & \downarrow&        & \downarrow \\
  \end{array}
\right) \,\,\in\mC^{N\times K} \, 
 \end{equation}
independent of the parameter vector 
$\vect \wl_q=[l_{1 q},l_{2 q},\dots,l_{K q}]^{\intercal}$ and thus fixed, and 
\begin{equation}\label{eq:matrixlambda}
\Lambda_{q} = 
\left(
\begin{array}{cccc}
    l_{1 q}& 0 & & \\
    0 &l_{2 q}& & \\
    & & \ddots & \\
    & & 0 & l_{K q}
 \end{array}
\right) \,\,\in\mC^{K\times K}\, 
 \end{equation}
diagonal. Then, under the assumptions that all sets of $M +1$ columns of $\tAc$ are linearly independent, 
 and the rank of the data matrix $B$ is $M$, 
 MUSIC provides the exact support of $\bfrho$ if the data are noiseless. 
\end{proposition}
\begin{remark}  
The assumption that rank of the data matrix $B$ is $M$ means that the excitations are sufficiently diverse, which is usually the case in practice. 
\end{remark}
\begin{proof}
All data vectors $\vect b_q$, $q=1,\ldots,\aleph$, are linear combinations of the same $M$ columns $\vect {\ta}_k$ of $\tAc$, indexed by $T=\supp(\vect \rho)$, with $M=|T|$.
Thus, the columns of $\tAc$ indexed by $T$ 
span a vector subspace of $\mC^N$ called the signal subspace. 
Furthermore, if all sets of $M  +1$ columns of $\tAc$ are linearly independent, 
no other column of $\tAc$ is contained in  the signal subspace 
in the noiseless case.
Hence, the unknown support $T$ is uniquely determined by the zero set of the
projections of the columns of $\tAc$ onto the noise subspace, 
which is the orthogonal complement to the signal  subspace. 
\end{proof}

The objective of the MUSIC algorithm is to find the support $T$ of an unknown sparse vector  $\vect \rho =[\rho_1, \rho_2, \dots, \rho_K]^{\intercal}$, when a number of nonzero entries $M$ is much smaller 
than its length $K$. With a sufficiently diverse number  of experiments $\aleph\ge M$ we create a data matrix
$\Bm$, and we compute its SVD
\begin{equation}\label{svdB}
B=U\Sigma V^\ast
=\sum_{j=1}^{K}\sigma_j\vect u_j\vect v_j^\ast\, .
\end{equation}
If the data are noiseless there are exactly $M$ nonzero singular values $\sigma_1>\sigma_2>\dots > \sigma_M>0$ with
corresponding left singular vectors $\vect u_j$, $j=1,\dots,M$, that span the signal subspace. 
The remaining singular values $\sigma_j$, $j=M+1,\dots,K$, are zero, and the corresponding left singular vectors 
span the noise subspace.  Since the set of  columns of $\tAc$ indexed by $T=\supp(\vect \rho)$ also spans the signal subspace,
the sought support $T$ 
corresponds to the zero set of the orthogonal projections
of the columns vectors  $\vect {\ta}_k$   
onto the noise subspace. Thus, 
it follows that the support of $\vect \rho$ can be found among  the  peaks of the imaging functional
\begin{equation}
\label{MUSIC}
\dsp \mathcal{I}_k^{\rm \sc MUSIC}=\frac{\|\vect {\ta}_k\|_{\ell_2}}{\sum_{j=M+1}^{N} | \langle \vect {\ta}_k, \vect u_j  \rangle |^2 }  \,, \,\, k=1,\dots,K.
\end{equation} 
In (\ref{MUSIC}), the numerator is a normalization factor. If all sets of $M+1$ columns of $\tAc$ are linearly independent, the peaks of~(\ref{MUSIC}) exactly coincide with the  support of $\vect \rho$.

Once the support of $\vect \rho$ is recovered, the problem (\ref{family1}) typically becomes overdetermined 
($N>M$) and the nonzero values of $\bfrho$ can be easily found by solving the linear system restricted to the given support with an $\ell_2$ or an $\ell_1$ method \cite{CMP16}.

 Consider imaging with noisy data. 
It follows from Weyl's theorem \cite{Weyl12} that when noise is added to the data 
 so $\Bm\rightarrow B^\delta$ with $\|B^\delta -B \|_{\ell_2}<\delta$, 
then no singular value $\sigma^\delta$ moves more than the norm of the perturbation, i.e., $\|\sigma^\delta -\sigma\|_{\ell_2}<\delta$. 
It follows that (i) perturbed and unperturbed singular values are paired, and (ii) the spectral gap between the zero and the nonzero singular values remains large 
if the smallest nonzero unperturbed singular value $\sigma_M\gg\delta$. Hence, if the noise is not too large,  we can determine the number of scatterers because it equals  the number of significant  
singular values of the data matrix $\Bm^\delta$.

The signal and noise subspaces are also perturbed in the presence of noise. It can be shown that the perturbed and unperturbed subspaces also remain close,  with changes proportional to the reciprocal of the spectral gap $\beta=\sigma^\delta_{M} - \sigma_{M+1}$ \cite{Wedin72}. 
We refer to \cite{Liao16}, and references therein, for a recent discussion about how much noise the 
MUSIC algorithm can tolerate.
Next, we give a result that states that MUSIC is robust provided
certain orthogonality conditions hold. 
For this theorem we introduce the parameter matrix
\begin{equation}
\label{eq:L}
L= 
\left(
\begin{array}{cccc}
   l_{11}& l_{12}&        & l_{1\aleph}\\
   l_{21}& l_{22}&        & l_{2\aleph}\\
    \vdots &  \vdots &     &  \vdots \\
   l_{K1}& l_{K2}&        & l_{K\aleph} 
  \end{array}
\right) 
\,\,\in\mC^{K\times \aleph} \,\, ,
\end{equation}
with which problem (\ref{family1}) can be rewritten as $ \tAc  X L = \Bm$, with $X =$Diag$(\vect \rho)$ (see (\ref{eq:AXL}) below).
 In order to formulate our next result we introduce the following notation.

\begin{definition}\label{submatrix}
Suppose $T=\supp(\vect{\rho})$. We denote by $X_T$ be the sub-matrix of $X$ where we keep the rows 
that correspond to $T$. Similarly,  we denote by $\vect y_T$  the sub-vector of any vector
$\vect y$ where we keep the entries that correspond to $T$.
\end{definition}
\begin{theorem}\label{th:musicnoise}
Assume $\vect \rho \in \mathbb{C}^{K}$ is $M$-sparse with $T=\supp(\vect \rho)$. 
Let $X =$Diag$(\vect \rho)$  be a diagonal matrix that solves
\begin{equation}\label{eq:AXL}
 \tAc  X L = \Bm,
\end{equation}
with $\Bm$ and $L$ given in (\ref{eq:datamatrix}) and (\ref{eq:L}), respectively. 
 Let  
\begin{equation}\label{ortho_cond}
\smT=\sigma_{\min}(L_T)
\end{equation}
be the minimal singular value of $L_T$.
Suppose the perturbed matrix $\Bm^\delta$ satisfies $\sigma_{\max}(B^\delta- B)\leq \delta$, and that the 
columns of $\tAc$ are normalized to one, that is  $\|  \vect {\ta}_i \|_{\ell_2} =1$ $\forall i$. 

If for some  $\eps <1/3$ 
the columns from the support of $\vect \rho$ satisfy the following approximate orthogonality condition
\begin{equation} \label{new_o}
  \forall  i, j \in T, \, i \neq j, |\langle  \vect { a}_i, \vect { a}_j \rangle |  < \frac{\eps}{M-1},
\end{equation}
and   $\delta$ is small so that  
  \begin{equation}\label{detection}
  2 \delta < \mu \,  \smT \, (1- 2 \epsilon),\quad \mbox{with} \quad \mu =\min_{\rho_i \neq 0}\{ |\rho_i| \},
  \end{equation}
then we can find a decomposition  $\Bm^\delta=Q^\delta + Q_0^\delta$ such that
 orthogonal projections onto the subspaces  $R(Q^\delta)$
 and $R(B)$ 
are close, so
\begin{equation}\label{close}
\|P_{R(Q^\delta)} - P_{R(B)}\|_{\ell_2} \leq\frac{\delta}{\mu \, \smT \, (1- 2 \epsilon)}.
\end{equation}
\end{theorem}

Theorem \ref{th:musicnoise} is, to the best of our knowledge, new. 
It gives conditions under which the perturbed and unperturbed subspaces remain close so MUSIC is robust with respect to additive noise. 
Note that Theorem \ref{th:musicnoise} allows the columns of $\tAc$ to be almost collinear as long as the columns that are in the support of the solution  are approximately orthogonal, 
so  (\ref{new_o}) holds. The fact that the error in the orthogonal projections (\ref{close}) is inversely proportional to the minimal singular value $\smT$ 
(see (\ref{ortho_cond})) can be interpreted as a quality control on the different sets of parameters $\vect \wl_q$ used to collect the data.
It says that MUSIC is not robust if these sets are chosen so that the data are not diverse enough so $\smT$ is small. 
In order for MUSIC to be robust the parameter vectors $\vect \wl_q$ that form the columns of $L$ should be as orthogonal as possible.
The proof of Theorem \ref{th:musicnoise} is given in \ref{app:proofmusic}. 

We also refer to  \cite{Lee12}  for  a subspace-augmented MUSIC algorithm that improves the performance of MUSIC under unfavorable conditions such as the lack of diversity of the data matrix.

\section{Data structures in active array imaging}  \label{sec:ds}
We consider here the active array imaging problem introduced in Section \ref{sec:array}. 
Our aim is to examine for which configurations the imaging problem can be written in the MMV form (\ref{eq:mmv}) so that MUSIC can be used.  It is known that MUSIC could be used 
successfully in two cases: either for fixed frequency data ($S=1$) and multiple transducers, or for a single transducer and multiple frequencies. We show that  
 a  factorization as in (\ref{factor}) can be obtained for these two cases in Subsections~\ref{sec:sfmr} and~\ref{sec:pronny}, respectively. We discuss these two cases in detail, because they are
  the building blocks of our construction for multiple frequencies and many transducers. We show in Subsection~\ref{sec:mus_mr_mf} how to 
  construct an approximate 
  MUSIC for multiple frequencies and many transducers. To the best of our knowledge, this is the first, albeit approximate, MUSIC algorithm for multiple frequencies and many transducers. The approximation 
  holds in the paraxial regime, when the array and the IW are small and the distance between them is large. We investigate numerically the quality of this approximation in Subsection~\ref{sec:data}, where we chose to use intensity-only measurements. This the most challenging type of data, that we consider in this work. In Subsection~\ref{sec:mus_mr_mf_nophases} 
  (and \ref{sec:phaseR}) we explain
  how this type of data  can be recast as a linear system  of the form~(\ref{eq:mmv}).

\subsection{Single frequency signals and multiple receivers}
\label{sec:sfmr}

Fix  a frequency $\omega$.  We denote  by $\vect\wf(\omega)=[\wf_1(\omega),\ldots,\wf_N(\omega)]^{\intercal}$ 
 the illumination vector whose entries are the signals sent from the corresponding sources $\vx_s$, $s=1,\dots,N$, on the array.  
 The most basic  illumination vectors are $\vect\we_{i}$, with all entries equal to zero except the $i$th entry which is 1. We will often use them in this work.
 Given an illumination $\vect \wf(\omega)$, our imaging data are
\begin{equation}\label{dataf}
\vect b_{\wf(\omega)}=\Pm (\omega)\vect \wf (\omega),
\end{equation}
 where $\Pm (\omega)$ is the single frequency response matrix (\ref{responsematrix}). These are the 
echoes recorded at the $N$ receivers  located at $\vx_r$, $r=1,\dots,N$, on the array.  

Let
$$\wg_{\wf(\omega)}^{(k)}=\vect\wg(\vy_k; \omega)^{\intercal}\vect\wf(\omega),\ k=1,\dots,K,$$ 
be the fields at the grid positions $\vy_k$ in the IW,  with $\vect\wg(\vy_k; \omega)$ given by (\ref{eq:GreenFuncVec}). Then, 
the data depend on the vector $\vect \wl=[\wg_{\wf (\omega)}^{(1)},\wg_{\wf (\omega)}^{(2)},\dots,\wg_{\wf (\omega)}^{(K)}]^{\intercal}$.
With a slight abuse of notation from Section \ref{sec:la}, we have indicated in (\ref{dataf}) that 
the control vectors are the illuminations $\vect \wf(\omega)$
instead of the vectors $\vect \wl$. The latter depend
on the Green's function vectors $\vect \wg(\vy;\omega)$ that are fixed by the physical layout, and on 
 the illumination vector $\vect\wf(\omega)$ that we control.

\begin{lemma}
Suppose the data $\vect b_{\wf(\omega)}$, corresponding to an illumination $\vect\wf(\omega)$
is obtained by
$$\vect b_{\wf(\omega)}=\Pm (\omega)\vect \wf (\omega)$$
Then 
\begin{equation}
\label{eq:justOne}
\bfb_{\wf(\omega)} = \Ac_{\wf(\omega)}\bfrho  \ ; \ \Ac_{\wf(\omega)} =  \tAc \, \Lambda_{\wf (\omega)}
\end{equation} 
 where 
\begin{equation}\label{eq:matrix1freqbis}
\tAc
 = 
\left(
\begin{array}{cccc}
    \uparrow & \uparrow&        & \uparrow \\
     \vect\wg(\vy_1; \omega)   &      \vect\wg(\vy_2; \omega)    &     \ldots &     \vect\wg(\vy_K; \omega)    \\
\downarrow & \downarrow&        & \downarrow \\
  \end{array}
  \right)
 \,\,\in\mC^{N\times K}, 
 \end{equation} 
and
\begin{equation}\label{eq:matrix1freqdiag}
\Lambda_{\wf (\omega)} = 
\left(
\begin{array}{cccc}
    \wg_{\wf(\omega)}^{(1)}& 0 & & \\
    0 & \wg_{\wf (\omega)}^{(2)}& & \\
    & & \ddots & \\
    & & 0 & \wg_{\wf (\omega)}^{(k)}
  \end{array}
  \right) \,\,\in\mC^{K\times K}. 
 \end{equation}
\end{lemma}

The proof of this Lemma immediately follows from the explicit formula
$$
 \Ac_{\wf(\omega)}
 = 
\left(
\begin{array}{cccc}
    \uparrow & \uparrow&        & \uparrow \\
    \wg_{\wf(\omega)}^{(1)} \vect\wg(\vy_1; \omega)   &     \wg_{\wf(\omega)}^{(2)} \vect\wg(\vy_2; \omega)    &  \ldots & 
    \wg_{\wf(\omega)}^{(K)} \vect\wg(\vy_K; \omega)    \\
\downarrow & \downarrow&        & \downarrow \\
  \end{array}
  \right) \,\,\in\mC^{N\times K}.
$$ 

A few remarks are now in order. The Lemma guarantees that for {\it any  family } $\bfb_{\wf_q(\omega)}$, $q=1, \dots, \aleph$, of  illuminations
 the decomposition
\begin{equation}
\label{eq:familyX}
\Ac_{\wf_q(\omega)}\bfrho = \bfb_{\wf_q(\omega)}
\end{equation} 
holds. Hence,  it follows from the discussion in Section \ref{sec:la} that 
 the support of $\bfrho$ can be found  with MUSIC exactly  if enough data  vectors $\vect b_{q}=\vect b_{\wf_q(\omega)}$ are available.
How to choose illuminations for these data vectors? A natural choice is to use the $\aleph=N$ illuminations 
$\vect\wf_q(\omega)=  \vect {e}_q$.
Then, the data-matrix is $B=\Pm (\omega),$  the single frequency response matrix~(\ref{responsematrix}). 
This is a  typical  choice in practice.

Secondly, in the noisy case the robustness of MUSIC depends  on $\smT$ defined in (\ref{ortho_cond})
as the minimum singular vector of the sub-matrix of $L$ with rows corresponding to the support of $\bfrho$.
 Let us investigate further this optimality 
for the single-frequency regime. Here, the illumination matrix is
\[
L= 
\left(
\begin{array}{cccc}
\uparrow & \uparrow&       & \uparrow \\
{\tAc}^{\intercal} \vect\wf_1(\omega) & 
 {\tAc}^{\intercal} \vect\wf_2(\omega) & \ldots & 
  {\tAc}^{\intercal} \vect\wf_{\aleph}(\omega) \\
  \downarrow & \downarrow&        & \downarrow \\
  \end{array}
\right) \in \mC^{K\times \aleph}\, .
\]
The $i$th column 
${ \tAc}^{\intercal} \vect\wf_i(\omega)=[\wg_{\wf_i(\omega)}^{(1)},\wg_{\wf_i(\omega)}^{(2)},\dots,\wg_{\wf_i(\omega)}^{(K)}]^{\intercal}$ of matrix $L$ contains the fields at all grid positions $\vy_k$, $k=1,\ldots,K$, due to illumination   $\vect\wf_i(\omega)$. If we use the $\aleph=N$ illuminations 
$\vect\wf_q(\omega)= f(\omega) \vect {e}_q$, then
$
L=  f(\omega) { \tAc}^{\intercal}.
$
Thus, assuming $\tAc$  satisfies the conditions of Theorem  \ref{th:musicnoise}, we get
$$
\smT=\sigma_{\rm min} (L_T) \ge (1-2\varepsilon) | f(\omega)|\,.
$$

\subsection{Multiple frequencies and one transducer: the one-dimensional problem}
\label{sec:pronny}

Consider a one-dimensional multifrequency imaging problem where we use only one transducer that works as source and  receiver.
Denote by  ${y}_n=L+(n-1) \Delta {y} $  the distance between the transducer and the scatterer of reflectivity $\rho_n$,  $n=1,\dots,K$.  Then, 
 \begin{equation}\label{fft_like}
 \sum_{n=1}^{K} e^{i  2  \kappa_m   {y}_n} \rho_n = b_m\, , \quad m=1,\dots,S,
 \end{equation}
relates the positions and reflectivities of the scatterers to the measurements $b_m$ at  frequencies $\omega_m=\kappa_m\,c_0$, where $c_0$ is the wave speed in a homogeneous medium.
In this problem, we seek to recover the unknown vector  $\bfrho=[\rho_1,\rho_2, \dots,\rho_{K}]$ from the multifrequency data vector 
$\vect b = [b_1, b_2, \dots, b_{S}]$ recorded at the single transducer. 

Problem (\ref{fft_like}) is well known in the signal processing literature as the estimation of signal parameters from a noisy exponential data sequence \cite{Stoicabook}. It can be solved efficiently with several methods, we refer for example to the SVD-prony method \cite{Kumaresan90} and the matrix pencil method \cite{Hua90}. We explain in this section how MUSIC can be used to find the solution for this one-dimensional imaging problem. In the next section we built upon this methodology to propose a multiple frequency MUSIC algorithm for the array imaging problem with many sources and many receivers.

We certainly can write
(\ref{fft_like}) in matrix form  $A  \bfrho=\vect b$, but we will only have one data vector $\vect b \in\mC^{S}$. 
The next assumption allows to elegantly formulate our data in the MMV format~(\ref{eq:mmv}) using 
a Prony-type argument~\cite{Prony} (see for example \cite{Griesmaier17}). 
Namely, suppose that the measurements are obtained at {\it  equally spaced } wavenumbers 
$\kappa_m = \kappa_1+ (m-1) \Delta \kappa$,   $m=1,2,\dots, S$, and let $S= 2 \aleph -1$. Then,  fill up the  
$\aleph \times \aleph$ data matrix $B$ as the square Toeplitz matrix
 \begin{equation}\label{1d_B}
 B=\left( \begin{array}{cccc}
         b_{1} &  b_{2} & \dots & b_{\aleph}\\
         b_{2} &  b_{3} & \dots & b_{\aleph+1} \\
         \dots &  \dots & \dots & \dots\\
          b_{\aleph} &  b_{\aleph+1} & \dots & b_{2 \aleph-1 }\\
        \end{array} \right).
 \end{equation} 
It is straightforward to verify the following claim.

\begin{lemma}
If $\vect b_q$ is the $q$th column of the matrix $\Bm$ in (\ref{1d_B}), then
$$
 \tAc \, \Lambda_q \, \vect \rho   = \vect b_q, \, q=1,2,\dots, \aleph,
$$
where
 \begin{equation}\label{1d_A}
  \tAc =  
\left( \begin{array}{cccc}
         e^{i  2  \kappa_1   {y}_1}   &  e^{i  2  \kappa_1   {y}_2}  & \dots & e^{i  2  \kappa_1   {y}_{K}} \\
         e^{i  2  \kappa_2   {y}_1}  &  e^{i  2  \kappa_2   {y}_2}  & \dots & e^{i  2  \kappa_2   {y}_{K}}  \\
         \dots &  \dots & \dots & \dots\\
          e^{i  2  \kappa_{\aleph}   {y}_1}  & e^{i  2  \kappa_{\aleph}   {y}_2}  & \dots & e^{i  2  \kappa_{\aleph}   {y}_{K}} 
        \end{array} \right)\, ,
 \end{equation}
 and the $K\times K$ diagonal matrices
$$  \Lambda_q =   \left(  \Lambda_1 \right)^q  , \,\, \hbox{ with } \Lambda_1 : =  \left( \begin{array}{ccccc}
         e^{i  2  \Delta \kappa   {y}_1}   & 0  & \dots &0& 0 \\
         0  &  e^{i  2  \Delta \kappa   {y}_2}  & \dots &0& 0  \\
         \dots &  \dots & \dots & e^{i  2  \Delta \kappa   {y}_{K-1}}  & 0\\
         0 & 0  & \dots & 0& e^{i  2  \Delta \kappa   {y}_{K}} \\
        \end{array} \right). 
 $$
\end{lemma}

As promised, we have obtained the desired structure of our data matrix $B$ for MUSIC to work. The key here was to stack the data in the cyclic fashion~(\ref{1d_B}). Such stacking worked because wavenumbers were equally spaced.
 Clearly, $B$ does not have to be square. As always, it  
 needs to have at least $M$ linearly independent columns for MUSIC to recover $M$ scatterers.

\subsection{Multiple frequency signals, multiple sources and receivers}
\label{sec:mus_mr_mf}
Finally, we consider the most general case in which multiple frequency signals 
are used to probe the medium using several transducers that emit and record them. 
This case considers all the possible diversity of information that
can be obtained from the illuminations.  
We discuss first the situation in which the receivers measure amplitudes and phases, and then the case in which they can only measure 
amplitudes squared.

The idea to stack data in the cyclic fashion~(\ref{1d_B}) motivated us to think whether there is a way to organize multiple frequency data that guarantees our decomposition
\begin{equation}\label{factor1}
\tAc \, \Lambda_q \, \vect \rho   = \vect b_q, \,\, q=1,2,\dots, \aleph.
\end{equation}
We were not able to find  an exact factorization~(\ref{factor1}) in general, and therefore, at present, MUSIC cannot be used to identify the support of $\bfrho$ exactly. 
 We claim, however, that factorization~(\ref{factor1}) is approximately valid in the  paraxial regime $\lambda\ll a\ll L$ if we choose
 \begin{equation}\label{Pstack}
B=\Pm^c : = [\Pm(\omega_1)^{\intercal}, \Pm(\omega_2)^{\intercal},\dots,\Pm(\omega_S)^{\intercal}]^{\intercal}\, ,
\end{equation}
 where  $\Pm (\omega_k)$ are the single frequency  $\omega_k$ response matrices~(\ref{responsematrix}). In this case 
$\aleph = N$,
 where $N$ is the number of transducers. Indeed, denote $\kappa_c=\omega_c/c_0$ as  the central wavenumber, 
 $\vy_j=(\vect y_j,L+\eta_j)$, and $\vx_s=(\vect x_s,0)$. Then, we have:

\begin{lemma} \label{lemma_parax}
Suppose we are in the paraxial regime, and the IW is small compared to $L$.
If $\vect b_q$ is the $q$th column of the matrix $\Bm$ in~(\ref{Pstack}), then
\begin{equation} \label{eq:mfmusic}
 \tAc_{q} \bfrho = \vect b_{q},\, \mbox{with}\,\, \tAc_q \approx \tAc\, \Lambda_{q}, \, q=1,\dots,\aleph,
\end{equation}
where $\tAc$ and $\Lambda_{q}$ are given by
\begin{equation}\label{eq:matrixSfreqhalfapprox}
 \tAc
 =
\left( 
\begin{array}{cccc}
    \uparrow & \uparrow&        & \uparrow \\
    \vect\wh(\vy_1; \omega_1)   & 
    \vect\wh(\vy_2; \omega_1)    & 
    \ldots & 
    \vect\wh(\vy_K; \omega_1)    \\
\downarrow & \downarrow&        & \downarrow \\
    \uparrow & \uparrow&        & \uparrow \\
    \vect\wh(\vy_1; \omega_2)   & 
    \vect\wh(\vy_2; \omega_2)    & 
    \ldots & 
    \vect\wh(\vy_K; \omega_2)    \\
\downarrow & \downarrow&        & \downarrow \\
\vdots & \vdots&        & \vdots \\
    \uparrow & \uparrow&        & \uparrow \\
     \vect\wh(\vy_1; \omega_{S})   & 
     \vect\wh(\vy_2; \omega_{S})    & 
    \ldots & 
    \vect\wh(\vy_K; \omega_{S})    \\
\downarrow & \downarrow&        & \downarrow
\end{array} \right) 
 \end{equation}
with $\vect\wh(\vy_j; \omega_l)=e^{i\kappa_l(L+\eta_j)} \vect\wg(\vy_j; \omega_l)$, and
\begin{equation}\label{eq:matrixSfreqdiag}
\Lambda_{q} = 
\left( \begin{array}{cccc}
    e^{i\kappa_c(\vect x_q - \vect y_1)^2/2L}& 0 & & \\
    0 & e^{i\kappa_c(\vect x_q - \vect y_2)^2/2L}& & \\
    & & \ddots & \\
    & & 0 & e^{i\kappa_c(\vect x_q - \vect y_K)^2/2L}
  \end{array} \right) \, .
 \end{equation}
 The approximation is of order $O\left( \frac{B a^2}{c_0 L} + \frac{\omega_c a^4}{c_0L^3} \right).$
\end{lemma}
\begin{proof} The proof of Lemma \ref{lemma_parax} is straightforward. We only outline the idea here. Assume we use an illumination ${\bf e}_q$, then the $j$th column of $\Ac_q$
is
\begin{equation}\label{eq:matrixSfreq}
\left( 
\begin{array}{c}
   \uparrow \\
     G (\vy_j,\vx_q; \omega_1) \vect\wg(\vy_j; \omega_1)    \\
\downarrow \\
  \uparrow \\
      G (\vy_j,\vx_q; \omega_2) \vect\wg(\vy_j; \omega_2)    \\
 \downarrow \\
 \vdots \\
 \uparrow \\
     G (\vy_j,\vx_q; \omega_S)  \vect\wg(\vy_j; \omega_{S})    \\
 \downarrow
\end{array} \right) \, ,
 \end{equation}
where $ G (\vy_j,\vx_q; \omega_l)$ is~(\ref{greenfunc}). Thus, if $L$ is much larger than $a$ and the IW is small
\[
G (\vy_j,\vx_q; \omega_l) = \dsp \frac{e^{\rmi\kappa_l|\vx_s-\vy_j|}}{4\pi|\vx_q-\vy_j|} \approx \frac{1}{ 4 \pi L} e^{ \rmi\kappa_l |\vx_q-\vy_j|}=
e^{i\kappa_l (L+\eta_j)}e^{i(\varphi + \tilde{\varphi})}\, ,
\]
with $\varphi=\kappa_c(\vect x_q-\vect y_j)^2/2L$ and $\tilde{\varphi}=O\left( \frac{B a^2}{c_0 L} + \frac{\omega_c a^4}{c_0 L^3} \right)$.
\end{proof}

Similar considerations imply that the factorization~(\ref{factor1}) works if illuminations satisfy $\vect \wf(\omega_l) = f(\omega_l) \vect \wf$. 
This means that the array uses the same illumination pattern $\vect \wf$  for all the frequencies. We do not discuss this case for simplicity of presentation.

It is natural to ask whether other approaches may be more fruitful. After all, we obtain only approximate MUSIC so perhaps one could have used instead an alternative data structure and obtain an exact MUSIC.
In our previous work~\cite{Moscoso17} we  tried to use
\begin{equation}\label{PDmultiple}
B=\Pm^d = 
 \left( 
\begin{array}{cccc}
 \Pm(\omega_1)& \ldots & 0& 0\\
0&  \Pm(\omega_2) & \ldots& 0\\
\ldots & \ldots &\dots &\dots\\
0&0&0&\Pm(\omega_S) \\
 \end{array} \right)\, 
 \end{equation}
to image with MUSIC.  We showed that imaging with such data structure is equivalent to imaging with
each frequency separately and summing up the resulting images incoherently. Therefore there is no significant
improvement over imaging with a single frequency if one uses (\ref{PDmultiple}) for imaging with MUSIC~\cite{Moscoso17}. 

\subsubsection{Imaging without phases}
\label{sec:mus_mr_mf_nophases}

In its classical form, the phase retrieval problem consists in finding a function from
the amplitude of its Fourier transform. In imaging, it consists in finding a vector $\bfrho$ that is compatible with a set of quadratic equations for measured amplitudes. This occurs in imaging regimes where only intensity data is recorded and, thus, most of the information encoded in the phases is lost.
Phase retrieval algorithms have been developed over a long time to deal with this problem \cite{GS72,Fienup82}. 
They are flexible and effective but depend on prior information about the image and can give uneven results. 
An alternative convex approach that guarantees exact recovery has been considered in \cite{CMP11,Candes13}, but its computational cost is extremely high 
when the problem is large. When, however, we control the illuminations we may recover the missing phase information 
using a completely different strategy. This strategy  was introduced in~\cite{Novikov14,Moscoso16,Moscoso17}. We explain here some of its aspects that are relevant to this work. 

Assume that only the intensities can be recorded at the array. In \ref{sec:phaseR} we show that, for a fixed receiver location,
we could recover single frequency cross correlated data from multiple intensity-only measurements.
On the other hand, as noted in \cite{Novikov14}, the support of the reflectivity $\bfrho$ can be recovered exactly by 
using the MUSIC algorithm on the single frequency  interferometric matrix $\Mm(\omega)=\Pm^*(\omega)\Pm(\omega)$ if the data are recorded at
several receivers.
For multiple frequencies, multiple sources and multiple receivers one can use the  data structure
\begin{equation}\label{Mmultiple}
B = \Mm^{c}:=
 \left( 
\begin{array}{c}
\Pm(\omega_1)^* \Pm(\omega_1)\\
\Pm(\omega_2)^* \Pm(\omega_1)\\
\vdots \\
\Pm(\omega_S)^*\Pm(\omega_1) 
 \end{array} \right)\, 
 \end{equation}
for pairs of frequencies $(\omega_l,\omega_1)$, $l=1,\dots,S$, to image coherently using MUSIC. Indeed, the matrix
$\Mm^{c}$ in (\ref{Mmultiple}) and the matrix $\Pm^{c}$  in (\ref{Pstack}) have the same column space
and, therefore, MUSIC can  form the images using  the SVD of $\Mm^{c}$ and the column vectors of matrix (\ref{eq:matrixSfreqhalfapprox}) as imaging vectors. 
We denote this data structure with the superscript $c$ to point out that we have stacked the one frequency
matrices $\Pm(\omega_l)$ and the two frequencies matrices $\Pm(\omega_l)^* \Pm(\omega_1)$ 
in a column.

\section{Numerical Simulations}
\label{sec:numerics}
We present here numerical simulations that illustrate the performance of MUSIC. 
The data are simulated using the model in (\ref{responsematrix}) with $\wG(\vx,\vy;\omega)$ as in (\ref{greenfunc}).
We first illustrate the relevance of Theorem \ref{th:musicnoise}  for active array imaging in the presence of noise, and then we discuss multifrequency imaging with phaseless data as it was explained in Subsection \ref{sec:mus_mr_mf_nophases}. 

\subsection{Imaging results in the framework of Theorem \ref{th:musicnoise}}
To study the robustness of MUSIC with respect to additive noise we consider in this section active array imaging with multiple sources and multiple receivers, but a single frequency; see subsection \ref{sec:sfmr}. Given a set of illuminations $\{\vect \wf_{q}(\omega) \}_{q=1,\dots,\aleph}$, the imaging problem is to determine the location and reflectivities of the scatterers from a data matrix $\Bm$ whose column vectors are given by (\ref{dataf}), including phases.
This problem admits an exact factorization of the form (\ref{factor}) and, therefore, MUSIC can be used for recovering the support of the solution. Furthermore, MUSIC provides the exact support of the reflectivity under the assumptions of Proposition \ref{prop1}. 

According to Theorem \ref{th:musicnoise}  the effectiveness of the illuminations can be characterized by  $\smT$ defined in (\ref{ortho_cond}). This parameter quantifies how well the support of the reflectivity is illuminated and, thus, it affects the robustness of the MUSIC results. Specifically, from (\ref{close}) the distance between the orthogonal projections onto the perturbed and unperturbed signal subspaces is inversely proportional to $\smT$ and, thus, a good set of illuminations is one for which $\smT$ is large. 

It was observed in \cite{CMP13, CMP14} that imaging using the  top  singular  vectors  of the data matrix as  illuminations
lowers  the  impact  of  the noise  in  the  data. These  illumination  vectors  are  optimal in  the  sense  that  they  result  in  array  data  
with  maximal  power,  which  is  proportional  to  the associated  singular  values. They can be computed systematically  from the  singular  value  decomposition  of  the  array  response  matrix (\ref{responsematrix}) if it is available,  or  with  an  iterative time  reversal  process, which is  a  very  efficient  acquisition  method  for  obtaining  the  essential part  of  the  array  response  matrix  as discussed in 
\cite{Prada95}.

It is easy to understand Theorem \ref{th:musicnoise}  when the  scatterers are well separated, meaning that the Green's function vectors 
$\vect \wg(\vy;\omega)$ evaluated on the support of the solution are approximately orthogonal.  Indeed, in this limit, the  top singular  
vectors  correspond  one-to-one  to  the  scatterers. Then, it follows that $\smT$ is optimal and close to $\| \vect\wg(\vz_j; \omega)\|^2$
evaluated at the  weakest  scatterer.

We plot in Figure \ref{fig:h4} the images obtained with MUSIC using different set of illuminations. The value of $\smT$ that corresponds to each set of illuminations is displayed above the images. The images are obtained in a homogeneous medium using an active array of $N = 81$ transducers that transmit and receive the signals. The frequency used is $600$ THz, corresponding to a wavelength $\lambda$ of $500$ nm (blue light). The array size is $100 \lambda$ and the distance from the array to the IW is $L=100\lambda$ as well. The IW is a rectangle  of size $5\lambda \times 50 \lambda$ discretized with a regular mesh of $50\times50$ rectangular elements. Different sets of illuminations are used to gather the data matrix $\Bm$. 
In all the figures, the true locations of the scatterers are indicated with white crosses, and the length scales are 
measured in units of  $\lambda_0$. In this numerical experiment, the scatterers are on the grid.
We add to the data mean zero uncorrelated noise corresponding to SNR~$=0$ dB.

The left most image of Figure \ref{fig:h4} shows the results obtained with MUSIC using optimal illuminations. We observe that MUSIC is very robust with respect to additive noise. The other three images are obtained with random illuminations: from top to bottom and from left to right the value of $\smT$ decreases. As expected from Theorem \ref{th:musicnoise} , the results are only good for
sets of illuminations with large $\smT$. Observe that MUSIC misses several scatterers in the two images in the bottom row of Figure \ref{fig:h4} corresponding to small $\smT$ values.

\begin{figure}[htbp]
\begin{center}
\begin{tabular}{cc}
$\smT=0.22$ & $\smT=0.16$ \\
\includegraphics[scale=0.28]{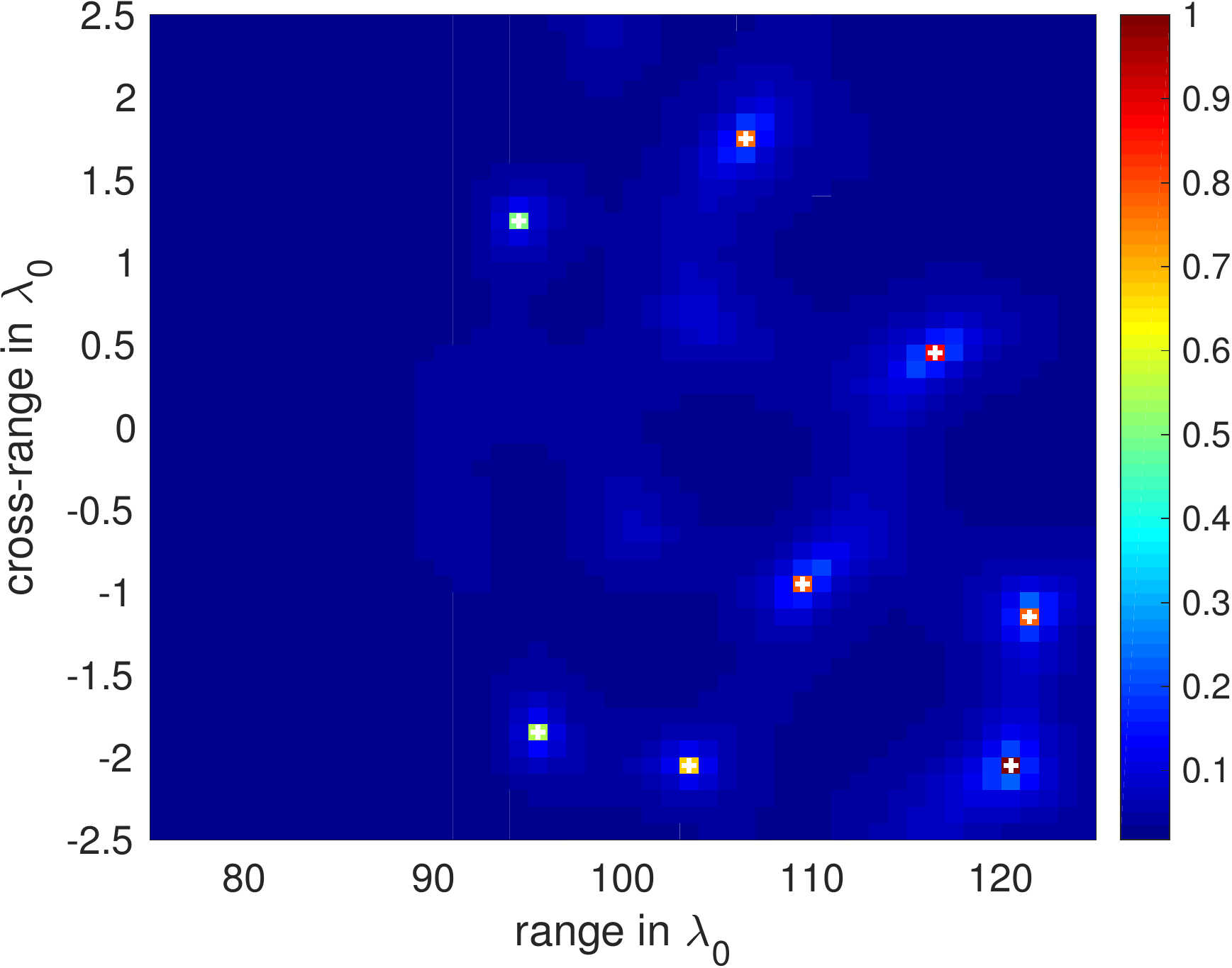}&
\includegraphics[scale=0.28]{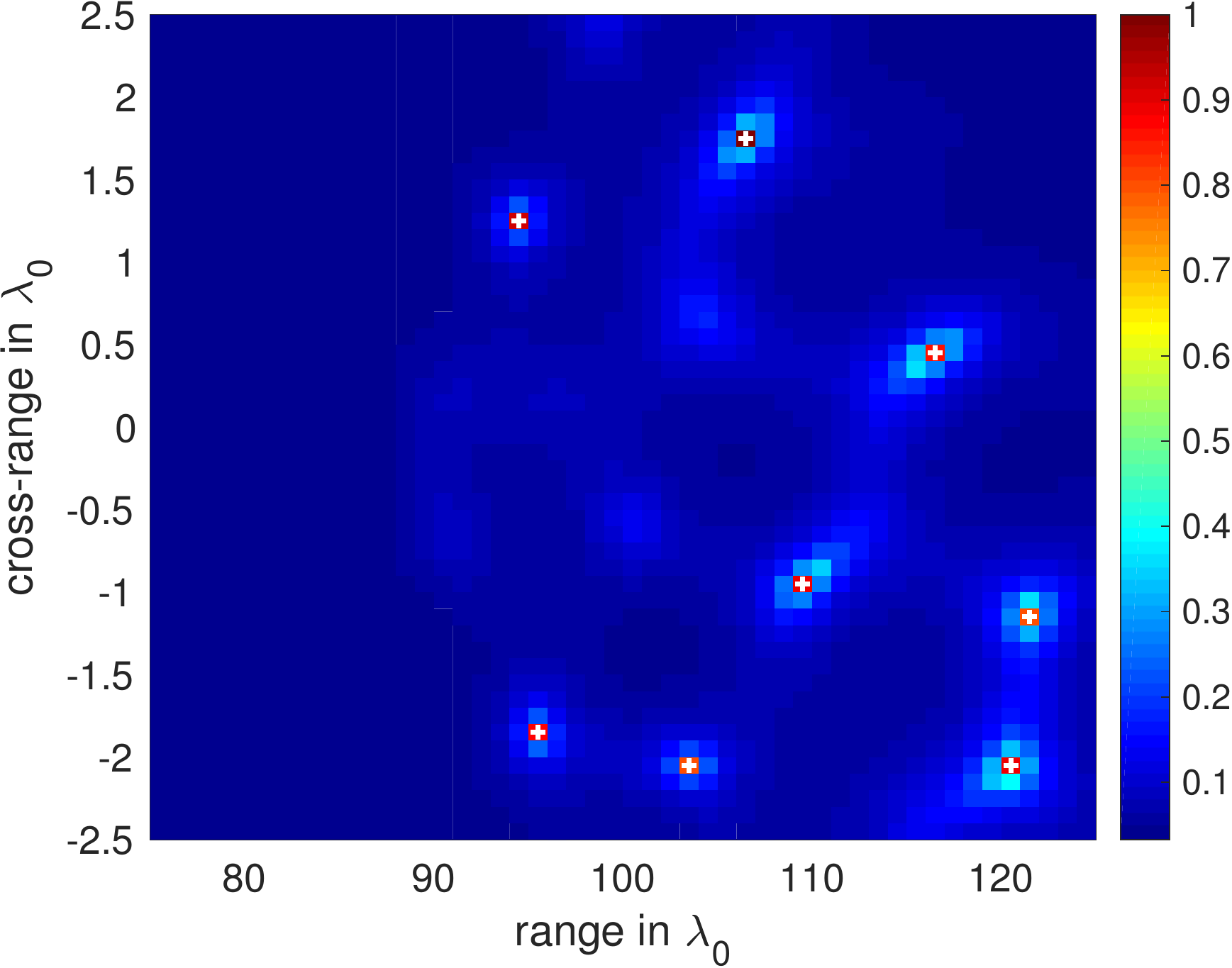}\\
$\smT=0.065$ & $\smT=0.064$\\
\includegraphics[scale=0.28]{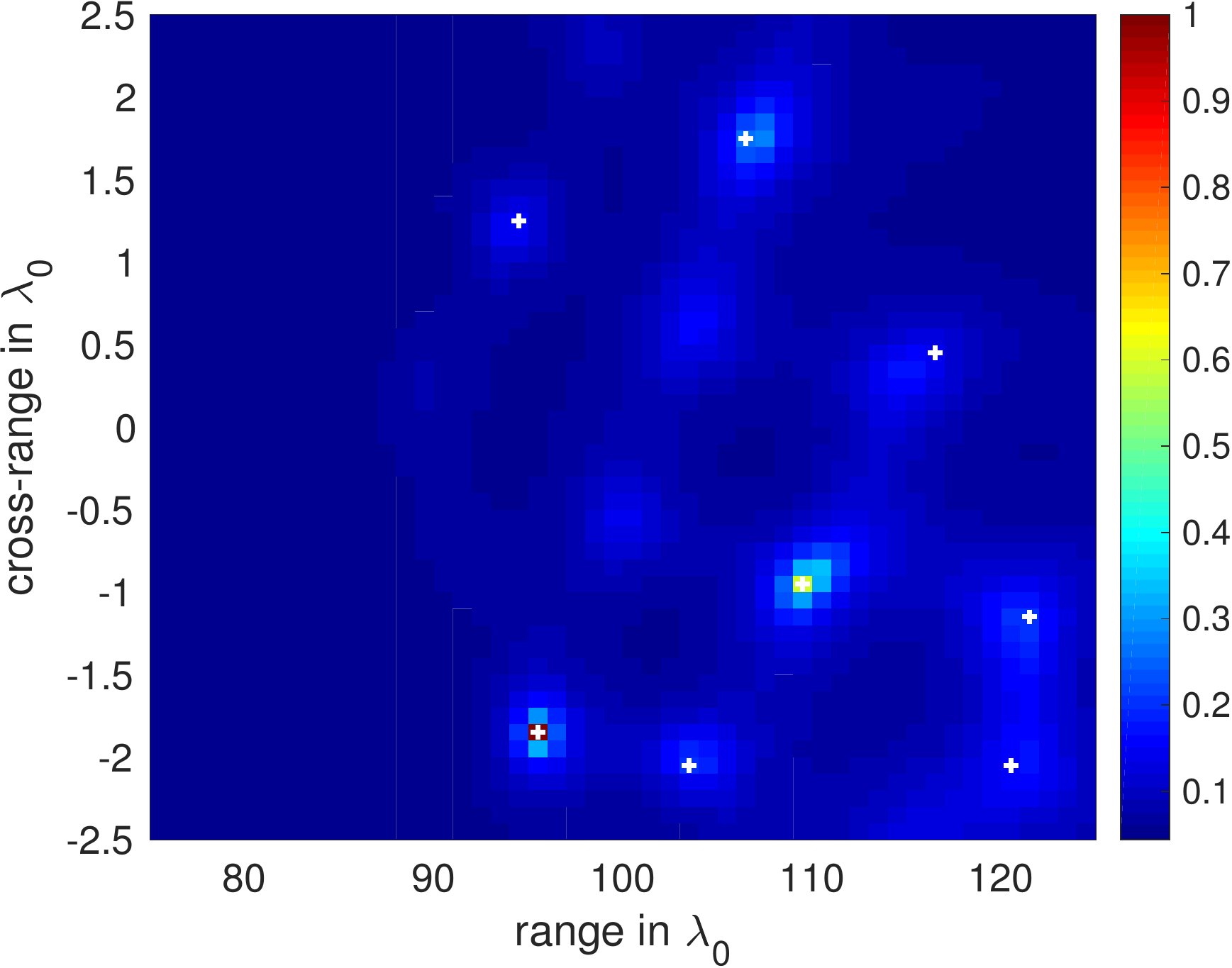} &
\includegraphics[scale=0.28]{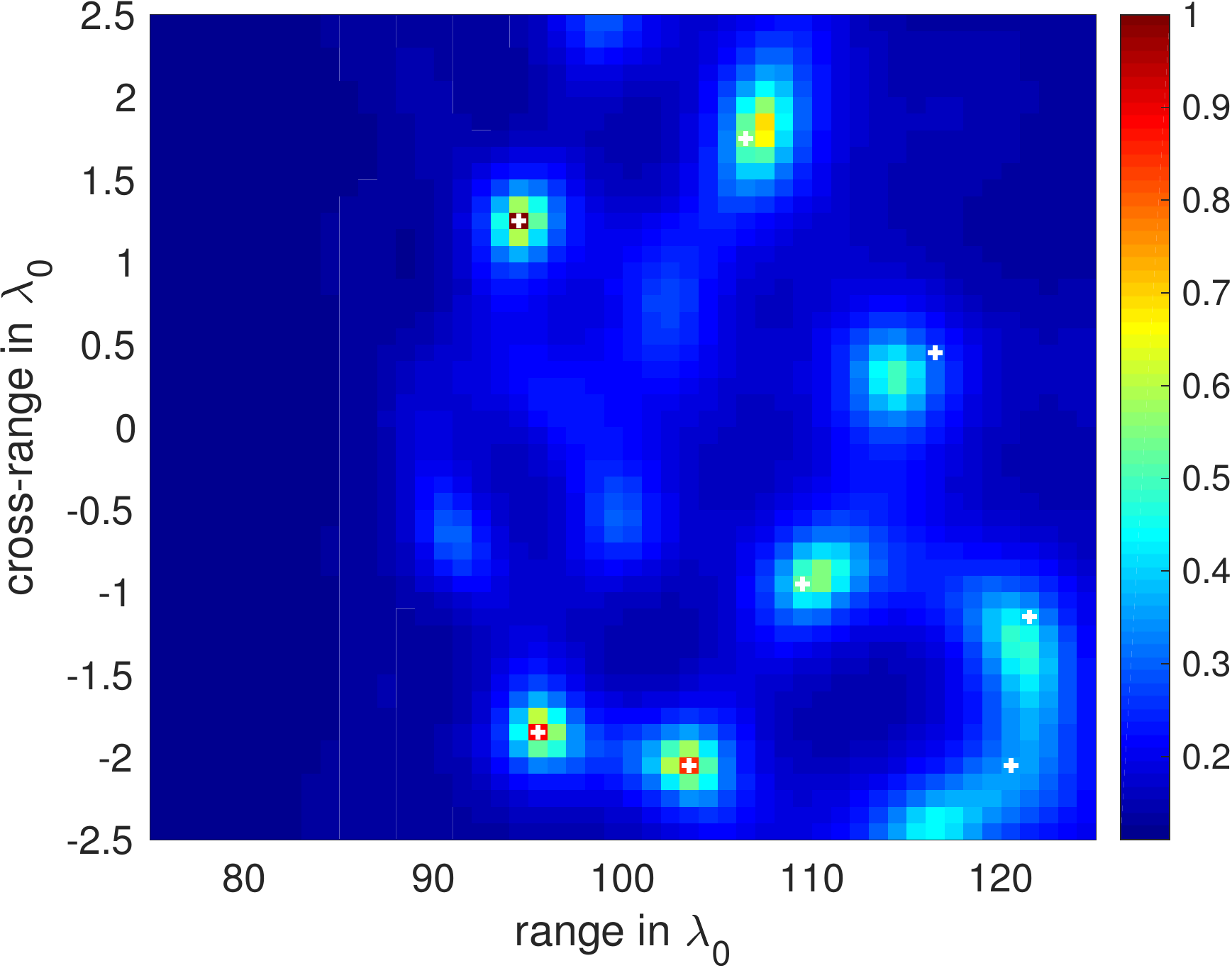} 
\end{tabular}
\end{center}
\caption{Imaging results using MUSIC with multiple sources and multiple receivers, but a single frequency.  SNR~$=0$dB corresponding to additive noise. 
The scatterers are on the grid. The top left  image is obtained using the optimal illuminations, for which $\smT=0.22$. The other three images are obtained using $12$ randomly chosen illuminations, for which the values of $\smT$ vary.
}
\label{fig:h4}
\end{figure}

\subsection{Multifrequency phaseless imaging}
\label{sec:data}
Next, we consider imaging with multiple sources, multiple receivers, and multiple frequencies, but phaseless data; see subsection \ref{sec:mus_mr_mf_nophases}. This case does not admit an exact factorization of the form (\ref{factor}) and, therefore, MUSIC does not provide the exact support of the solution. Still, it can  be used to estimate the support in the paraxial regime, when the scatterers are very far from the array and the IW  is small. Next, we examine numerically the deterioration of the resolution provided by MUSIC as the IW gets closer to the array.

We consider a central frequency $f_0=600$THz, typically used in optics, corresponding to a central wavelength $\lambda_0=500$nm.   
We use $S=12$  equally spaced frequencies covering a total bandwidth of $30$THz.  
All considered wavelengths are in the visible spectrum of green light.  
The size of the  array is $a=500\lambda_0$, and the distance between the array and the IW is $L=10000\lambda_0$. 
The medium between the array and the IW is homogeneous. The IW, whose size is $100 \lambda_0 \times 100 \lambda_0$,  is discretized
using a uniform lattice with mesh size $2\lambda_0\times 2\lambda_0$. Thus, the unknown image has $51\times51$ pixels.   
For this imaging system, we expect the cross-range and range resolutions to be of the order of $\lambda_0 L/a=20 \lambda_0$ and $C_0/B = \lambda_0 f_0/B=20\lambda_0$, respectively. 
In this setup, the propagation distance $L$ is large, and the array and the IW sizes are small so that the paraxial approximation holds. 

We assume that the phases of the signals received at the array cannot be measured. Hence, only their intensities are available for imaging. 
These measurements
are collected at multiple receivers, so we use the methods explained in subsection \ref{sec:mus_mr_mf_nophases} to image interferometrically.
\begin{figure}[htbp]
\begin{center}
\begin{tabular}{cc}
\includegraphics[scale=0.28]{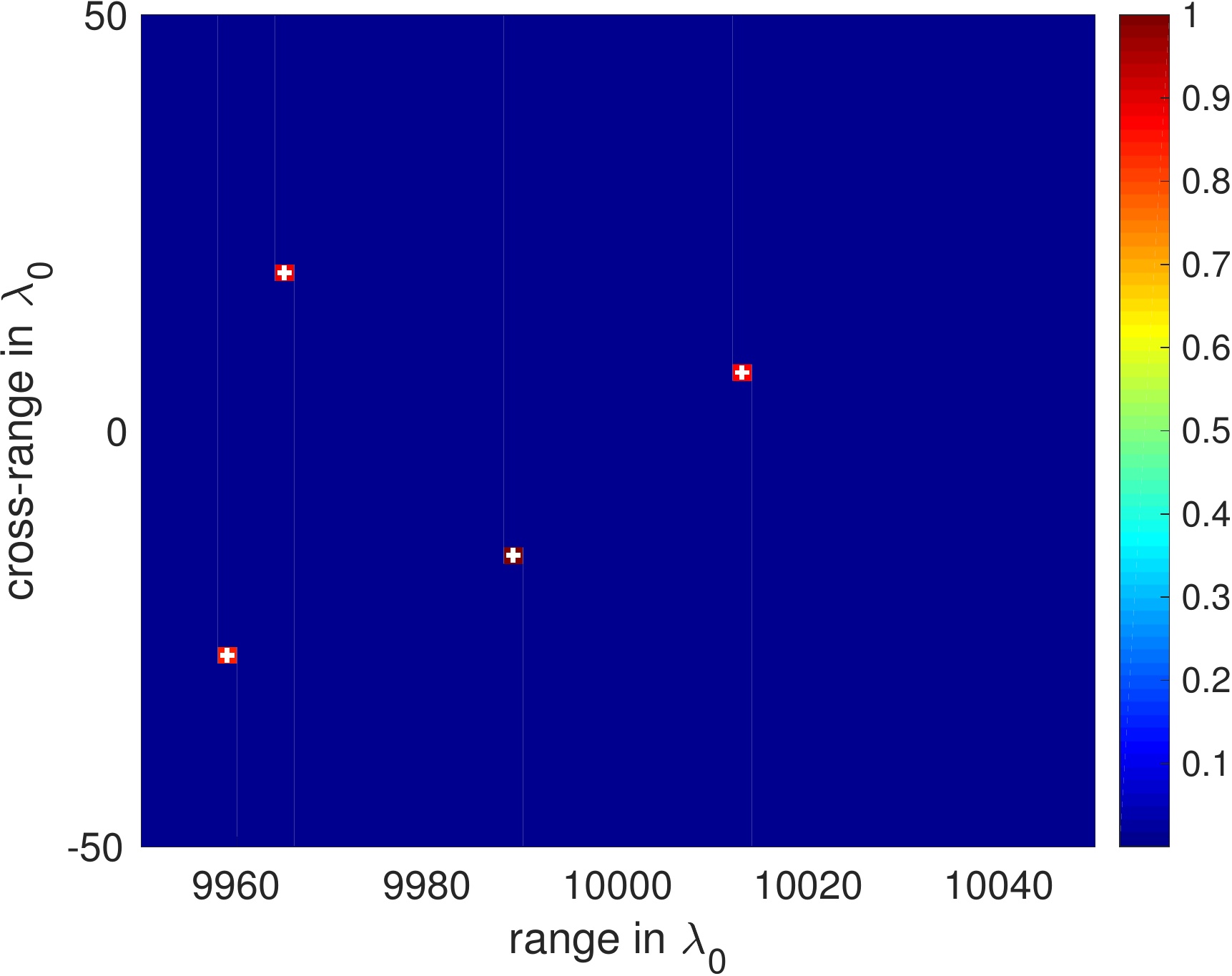}&
\includegraphics[scale=0.28]{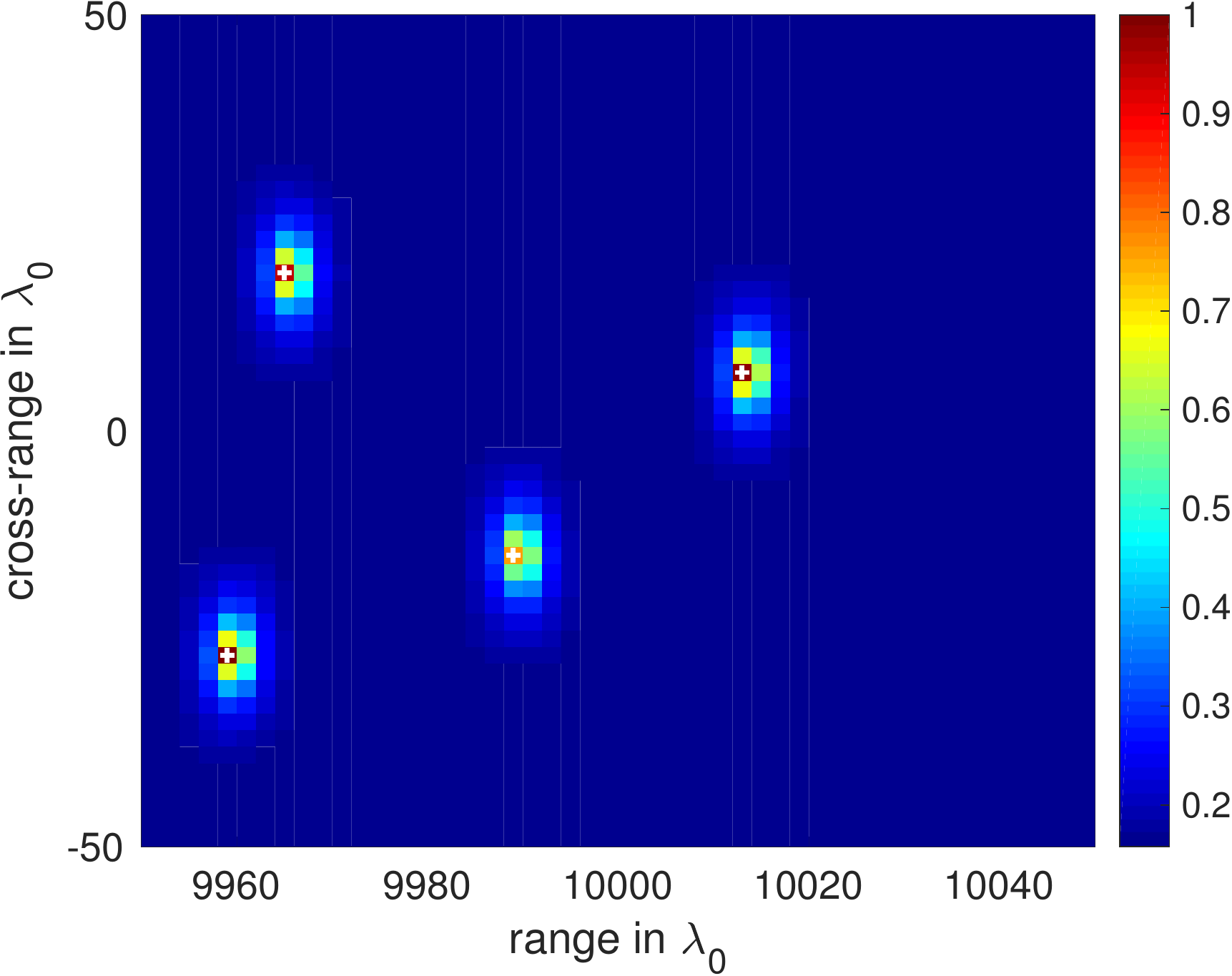}
\end{tabular}
\end{center}
\caption{There is no noise added to the data and the scatterers are on the grid. The left panel is the image constructed using 
MUSIC with $\Mm^d$. The right panel is obtained using MUSIC with $\Mm^{c}$ that couples the data over frequencies.}
\label{fig:h1}
\end{figure}

In Figure \ref{fig:h1}, the scatterers lie on the grid and no noise is added to the data. Hence, the data are exact.
We observe that imaging with MUSIC using the block-diagonal matrix $\Mm^d$ (left image)  gives exact recovery, while
MUSIC using the  $\Mm^{c}$ matrix (right image) that  couples all the frequencies is less accurate. This is so because, as we explained in Section \ref{sec:mus_mr_mf}, MUSIC with  $\Mm^{c}$  is not exact as it only provides,  in the paraxial regime, approximate locations of the scatterers. 

Figure~\ref{fig:h2} shows the same experiment as Figure~\ref{fig:h1}  but with off-grid scatterers. In this figure, the scatterers are displaced by half the grid size with respect to the grid points in both range and cross-range directions. 
This produces perturbations in the unknown phases of the signals collected at the array due to modeling errors. We remark that although the phases are not directly measured they are encoded in the  intensity measurements. We observe in Figure~\ref{fig:h2} that the image obtained with MUSIC using the $\Mm^d$ data 
structure  (left plot) deteriorates dramatically because the multiple-frequency information contained in the data is not processed in a coherent way. On the other hand, 
MUSIC with the $\Mm^{c}$ data structure (right plot) is very robust with respect to the off-grid displacements.

\begin{figure}[htbp]
\begin{center}
\begin{tabular}{cc}
\includegraphics[scale=0.28]{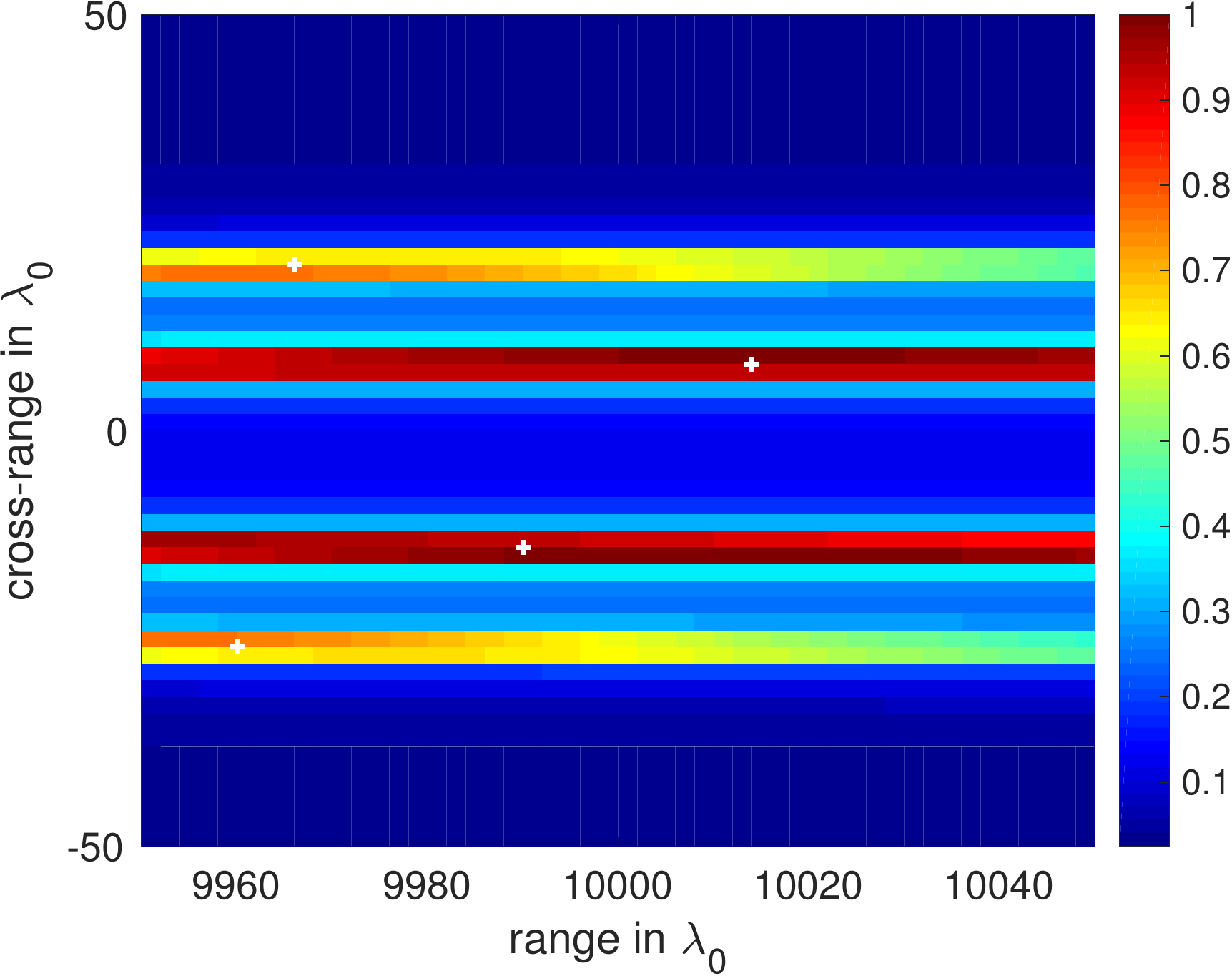}&
\includegraphics[scale=0.28]{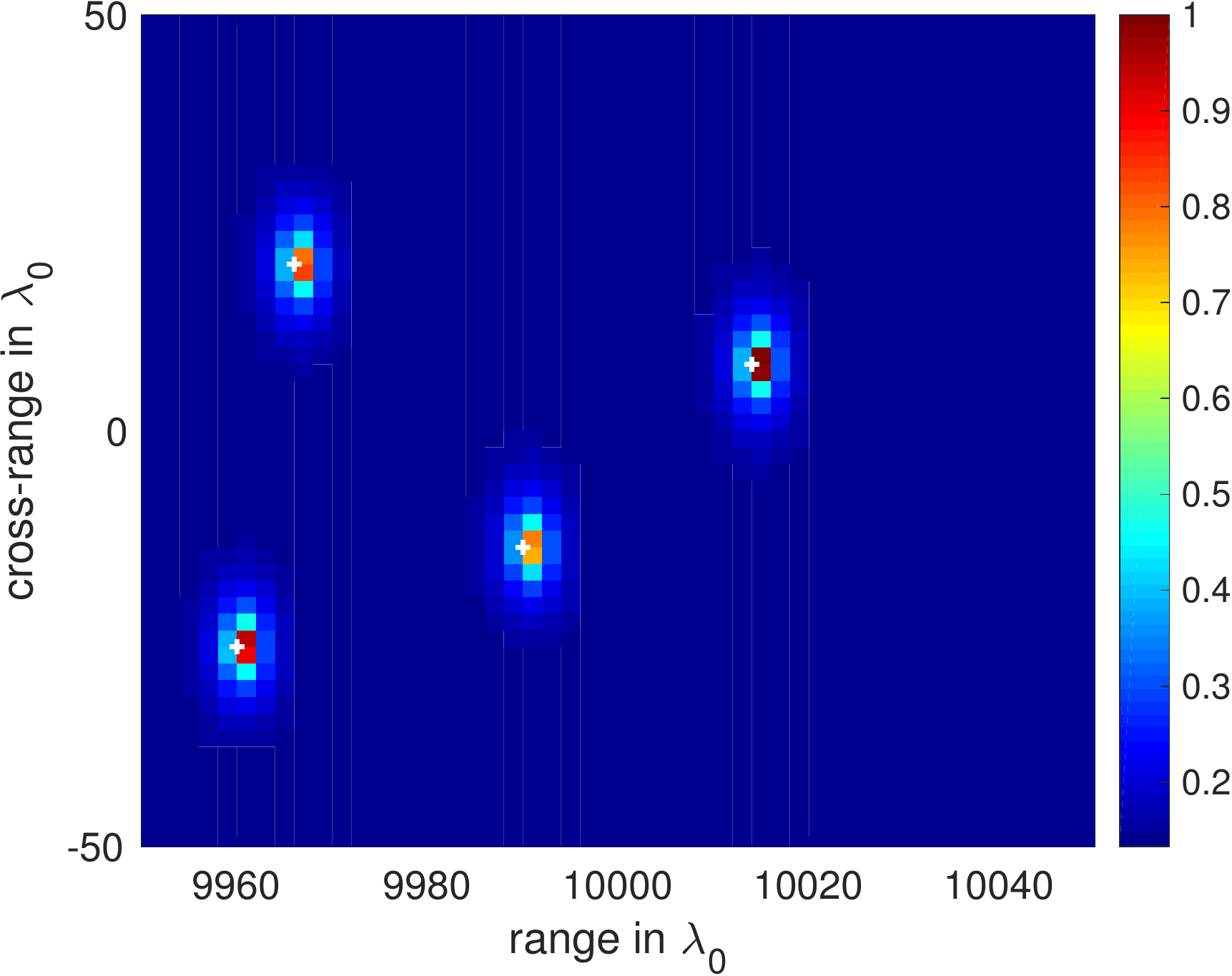}
\end{tabular}
\end{center}
\caption{Same as Figure~\ref{fig:h1} but with the scatterers off the grid. The scatterers are displaced by half the grid size in both directions from a grid point.}
\label{fig:h2}
\end{figure}

As noted above, multifrequency MUSIC using the matrix $\Mm^{c}$ is not exact. It only gives an approximation to the support of the scatterers  in the paraxial regime.  Thus, we expect  the resolution to improve (resp.  deteriorate) as the IW is moved  further (resp. closer) from the array. To examine its accuracy, we consider  in 
Figure \ref{fig:h3} imaging configurations with different ratios $a/L$.  We display from left to right the results for $a/L$ equal to $1/100$, $1/20$, $1/4$ and $1$. For a meaningful comparison, the mesh size in cross-range is adjusted so that it is always one tenth of the nominal resolution $\lambda_0 L/a$, i.e.,
the mesh size in cross-range is $\lambda_0 L/(10a)$ in all the images shown in Figure \ref{fig:h3}. In order words, the number of pixels in the images 
is kept constant
by changing the sizes of the IWs according to the relation $5 \lambda_0 L/a \times 5 (C_0/B)$. Thus, all the images in Figure \ref{fig:h3} have $51\times51$ pixels. As expected, the images in this figure show an almost exact recovery for small $a/L$ ratios and a worsening of the results as the ratio increases.

\begin{figure}[htbp]
\begin{center}
\begin{tabular}{cc}
$a/L=0.01$ & $a/L=0.05$ \\
\includegraphics[scale=0.28]{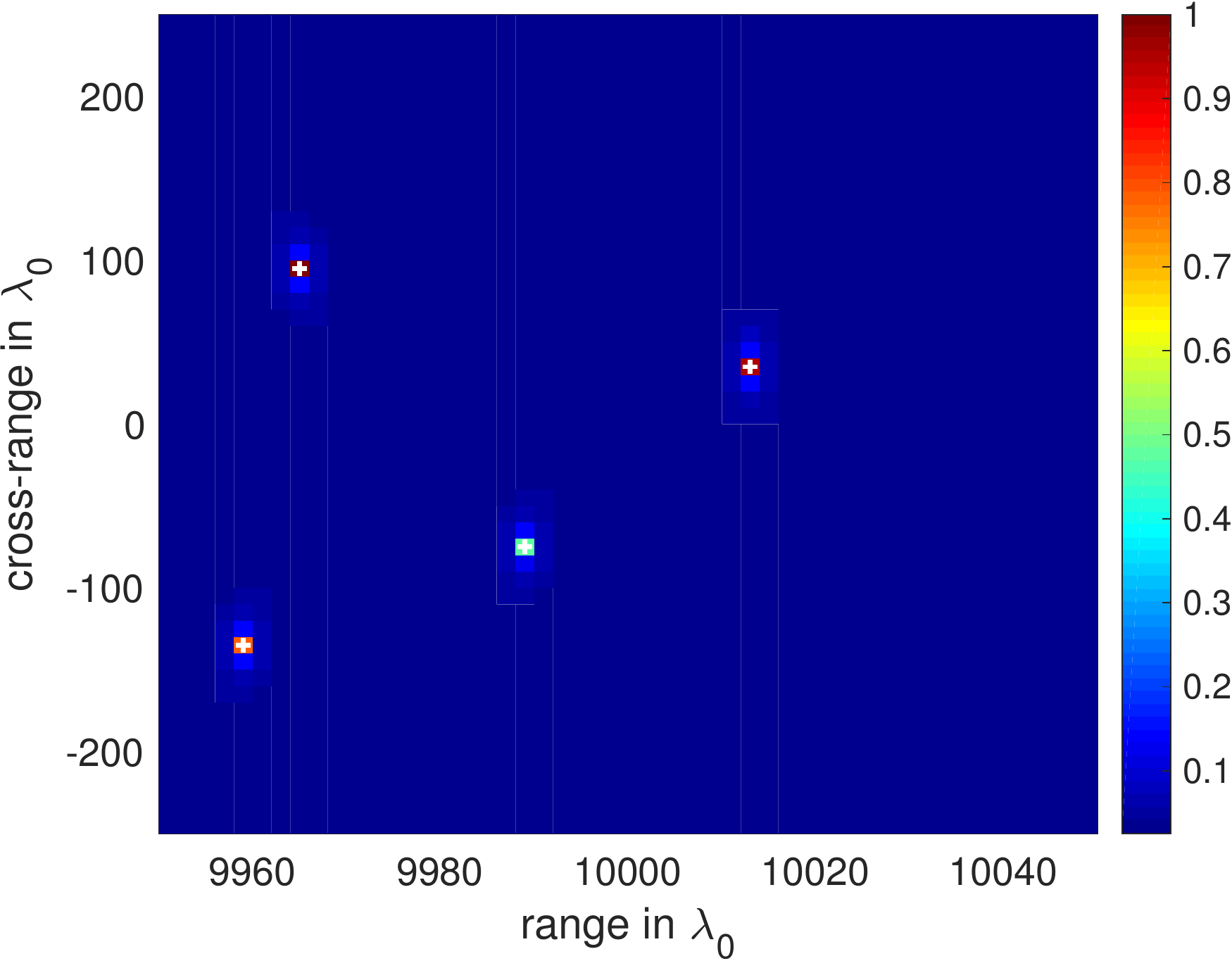}&
\includegraphics[scale=0.28]{IM_MUSIC_NEW_real1On} \\
$a/L=0.25$ & $a/L=1$\\
\includegraphics[scale=0.28]{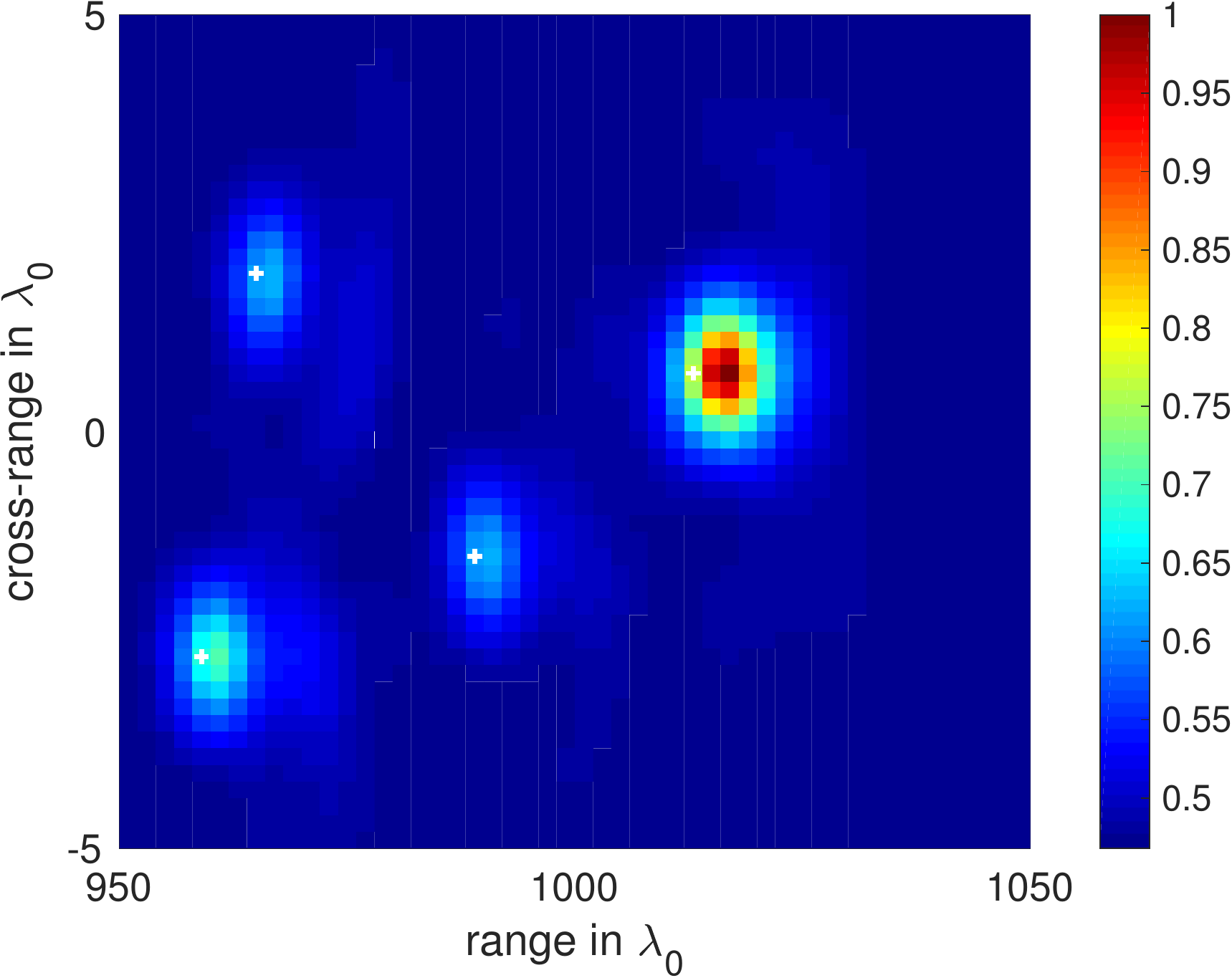} &
\includegraphics[scale=0.28]{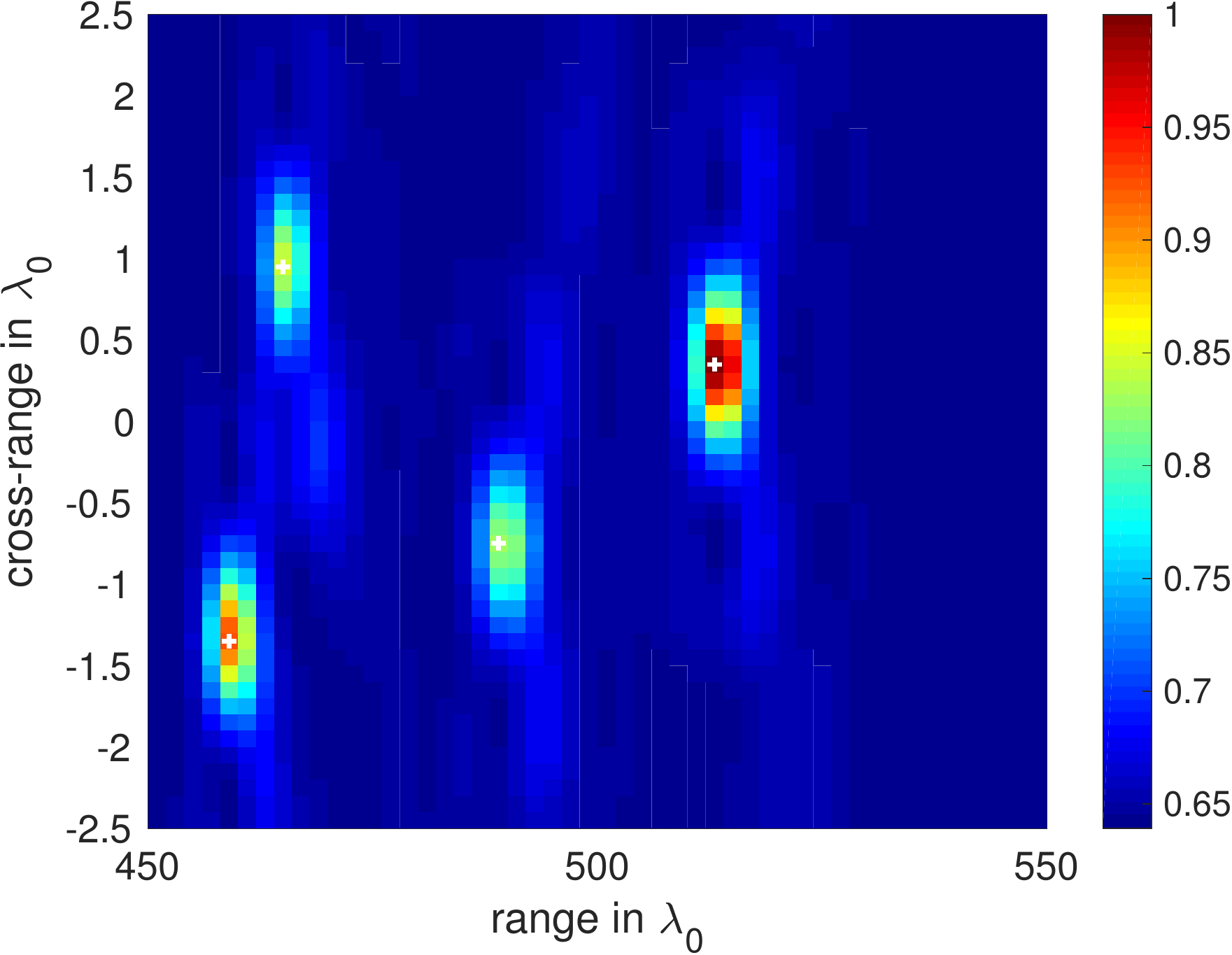} \\
\end{tabular}
\end{center}
\caption{Imaging results using MUSIC with $\Mm^{c}$ coupling over frequencies. From left to right and top to bottom the ratio $a/L$ increases and, therefore, the error due to the paraxial approximation increases so the accuracy of the MUSIC reconstruction decreases. The scatterers are on the grid.}
\label{fig:h3}
\end{figure}

\section{Conclusions} \label{sec:concl}
In this paper we discussed appropriate data structures that allow robust images with MUSIC, a method that is well adapted to finding sparse solutions of linear underdetermined systems of equations of the form 
$\Ac_{\wl_q}\bfrho=\bfb_{\wl_q}$.  In this work $\bfrho$ is the reflectivity, the image that we want to form, and $\vect \wl_q$ is a parameter vector that can be varied, such as the illumination profile of 
the imaging system 
in space and/or frequency. Given the data $\bfb_{\wl_q}$, our first main result  is the key observation that MUSIC provides the exact support of the unknown $\bfrho$ when the matrix $\Ac_{\wl_q}$ admits a factorization 
of the form $\Ac_{\wl_q}= \tAc \,\Lambda_{\wl_q}$ with $\Lambda_{\wl_q}$ diagonal. We also show in Theorem \ref{th:musicnoise} that MUSIC is robust with respect to noise provided the diversity of the data is high enough. Our second main contribution is 
 an approximate MUSIC algorithm for  multifrequency and multiple receiver imaging which is
obtained  under the paraxial approximation. Its robustness is illustrated with numerical simulations in an optical digital microscopy imaging regime.

\section*{Acknowledgments}
Part of this material is based upon work supported by the National Science Foundation under Grant No. DMS-1439786 while the authors were in 
residence at the Institute for Computational and Experimental Research in Mathematics (ICERM) in Providence, RI, during the Fall 2017 semester.
The work of M. Moscoso was partially supported by Spanish grant FIS2016-77892-R. The work of A.Novikov was partially supported by NSF grant DMS-1813943. The work of C. Tsogka was partially supported by AFOSR FA9550-17-1-0238.

\clearpage
\appendix

\section{Proof of theorem \ref{th:musicnoise}}
\label{app:proofmusic}

\begin{proof} 
 We claim that
\begin{equation}\label{yo_yo}
 (1-  { 2} \, \epsilon)^2 \| \vect  z \|_{\ell_2}^2 \leq \| ( \tAc^* \vect z)_T\|_{\ell_2}^2 \leq  (1+ { 2} \, \epsilon)^2 \| \vect  z \|_{\ell_2}^2
\end{equation}
if $\vect  z \in R(B)$ and $\eps<1/3$. Indeed, suppose that
\[
\vect  z= \sum_{i \in T} \alpha_i \vect a_i. 
\]
Then, defining $\vect \alpha$ 
as the vector in $\mC^K$ whose components are zero except the $i$th components with $i\in T$ that are equal to $\alpha_i$,
we get
\[
\left| \| \vect z\|_{\ell_2}^2-  \| \vect \alpha \|_{\ell_2}^2 \right|= \left|  \sum_{i, j \in T, i \neq j}  \bar{\alpha_i} \alpha_j \langle \vect a_i, \vect a_j \rangle \right| \leq \epsilon  \| \vect \alpha \|^2_{\ell_2}\, ,
\]
and
\[
  (1- \epsilon) \| \vect \alpha \|^2_{\ell_2} \leq \| z\|^2_{\ell_2} \leq  (1+ \epsilon) \| \vect \alpha \|^2_{\ell_2}.
\]
For any $j \in T$ we have
\[  
(\tAc^* \vect z)_j =  \sum_{i \in T} \alpha_i \langle \vect a_j, \vect a_i \rangle\, ,
\]
and, therefore,
\[
\| (\tAc^* \vect z)_T\|^2_{\ell_2} = 
 \sum_{ i,j,k \in T}  \bar{\alpha_j}  \alpha_i \langle \vect a_k, \vect a_i \rangle  \overline{\langle \vect a_k, \vect a_j \rangle}  \, .
\]
Hence,
\[
\left| \| ( \tAc^* \vect z)_T\|_{\ell_2}^2
-  \| \vect \alpha \|^2_{\ell_2} \right|  = 
 \left|   \sum_{ j, k \in T, j \neq k}   \left| \alpha_j \right|^2 ~ \left| \langle \vect a_k, \vect a_j \rangle \right|^2  +  \sum_{ i, j, k \in T, i\neq j}    \bar{\alpha}_j  \alpha_i \langle \vect a_k, \vect a_i \rangle  \overline{\langle \vect a_k, \vect a_j \rangle}   \right|
  \]
  \[
  \leq   \frac{\eps^2}{M-1} \| \vect \alpha \|^2_{\ell_2}+ \!\!\! 
  \sum_{ i, j \in T, i\neq j} \frac{|\alpha_j|^2 + |\alpha_i|^2}{2}  \left( \frac{2 \eps}{M-1}+ \frac{\eps^2 (M-2)}{(M-1)^2} \right)  \leq  (2 \eps + \eps^2) \| \vect \alpha \|^2_{\ell_2}.
\]
Therefore,
\[
 (1 -  2\, \epsilon - \eps^2) \| \vect \alpha \|^2_{\ell_2} \leq \| ( \tAc^* \vect z)_T\|^2_{\ell_2} \leq  (1+  \epsilon)^2 \|\vect  \alpha \|^2_{\ell_2}, 
\]
and we obtain
\[
 \frac{1-  2\, \epsilon - \eps^2}{1 + \eps} \| \vect z \|^2_{\ell_2} \leq \| (\tAc^* \vect z)_T\|^2_{\ell_2} \leq  \frac{(1+  \epsilon)^2}{1 - \eps} \| \vect z \|_{\ell_2}^2, 
\]
which implies~(\ref{yo_yo}) if $\eps <1/3$\footnote{{ This is an overestimate. It suffices to have $\eps - \eps^2 - 4 \eps^3 >0$. }}. 

In order to compute the smallest nonzero singular value of $\Bm$ we observe that
\[
\min_{\vect z \in R(B), ||\vect z||_{\ell_2} =1} \vect z^* B B^* \vect z 
=\min_{\vect z \in R(B), ||\vect z||_{\ell_2} =1}  ( \tAc^* \vect z)^*_T X_T L_T  L^*_T \bar{X}_T ( \tAc^* \vect z)_T \]
\[
\geq (1-{ 2} \epsilon)^2  \min_{\vect y \in \mathbb{C}^M ||\vect y||_{\ell_2} =1}  \vect y^* X_T L_T  L^*_T \bar{X}_T \vect y
 \geq (1-{ 2} \epsilon)^2 { \mu^2} (\smT)^2\, ,
\]
where we have used 
{
that $\smT$ is the smallest singular value of $L$. 
}
Since 
$\sigma_{\max} (B^\delta- B)  \leq \delta$, we conclude that $\Bm^\delta=Q^\delta+Q^\delta_0$, where $Q^\delta$ has $M$ nonzero singular values,
with smallest nonzero singular value
\[
\sigma_{\min}( Q^\delta) \geq \mu \smT (1- 2 \epsilon) - \delta\, ,
\]
and $Q^\delta_0$  has largest singular value
$\sigma_{\max}( Q^\delta_0) \leq \delta$.
If $ 2 \delta < \mu  \smT (1- 2 \epsilon)$, then we can discard $Q^\delta_0$ by truncation of the singular values smaller than the noise level. 
We now apply Wedin Theorem \cite{Wedin72} (see Theorem~\ref{wed}  below)
{
to obtain
\[
\|P_{R(Q^\delta)} - P_{R(B)}\|_{\ell_2} \leq\frac{\delta}{\mu \smT (1- 2 \epsilon)}.
\]
}
\end{proof}

\begin{theorem}\label{wed} (Wedin)
{
Let $\Bm=Q+Q_0$, where $Q$  has the SVD  $Q=U\Sigma V^{\intercal}$, and consider the perturbed matrix $\Bm^\delta=B + E$. 
If there exists a decomposition $\Bm^\delta= Q^\delta+Q^\delta_0$, and 
 two constants $\alpha\ge 0$ and $\beta > 0$ such that largest singular value $\sigma_{max}(Q_0)\le\alpha$ and smallest singular value $\sigma_{min}(Q^{\delta})\ge\alpha+\beta$, }
then the distance between  the orthogonal projections onto the subspaces  $R(Q)$ and $R(Q^\delta)$ is bounded by
\begin{equation}
\|P_{R(Q^\delta)} - P_{R(Q)}\|_{\ell_2} \leq\frac{\delta}{\beta}\, ,
\end{equation}
where $\delta = \max(\|EV\|_{\ell_2}, \|E^*U\|_{\ell_2})$. 
\end{theorem}

\section{The single frequency phase retrieval problem}
\label{sec:phaseR}%
 We consider here the same imaging configuration as in subsection \ref{sec:sfmr}, where signals  of only one frequency $\omega$ are sent from an array of transducers that emit  and record the signals. However, we assume now that only the intensities of the signals can be measured, so only the amplitudes square of the data vectors $\bfb_{q}=  \tAc  \bfrho_{q}$ are recorded.
Then, the  phase retrieval problem is to find the unknown vector $\bfrho$ from 
the family of quadratic equations 
\begin{equation}
\label{eq:phretrieval}
|{ \tAc}  \bfrho_{q}|^2= |\bfb_{q}|^2\, , \quad q=1,\dots,\aleph,
\end{equation}
{ where $|\cdot|$ is } understood component wise.

\subsection{A single receiver}

Problem (\ref{eq:phretrieval}) is
nonlinear and nonconvex and, hence, difficult to solve.
In fact, it is in general NP hard \cite{Sanz85}.
However, if an appropriate set of illuminations is used, we can take advantage of the polarization identity 
\begin{eqnarray}
2 \,\mbox{Re} < u, v> &=& |u+ v|^2 - |u|^2 - |v|^2 \nonumber\\
2 \,\mbox{Im} < u, v> &=& |u - i v|^2  - |u|^2 -  |v|^2 \,\, 
\label{eq:polariden}
\end{eqnarray}
to solve simple linear systems of the form
\begin{equation}\label{familyInt0}
{ \tAc}\,  \bfrho_{q} = \bfm^{(r)}_{q}\, , \quad q=1,\dots,\aleph,
\end{equation}
for a fixed receiver location $\vect x_r$. The polarization identity  allows us to find the inner product between two complex numbers and, therefore, its phase differences.
In (\ref{familyInt0}), $\bfm^{(r)}_{q}$ is the vector whose $i$th component is 
the correlation $\overline{b^{(r)}_{q}}b^{(r)}_{{ e}_i}$ between two signals measured at the  receiver
$\vx_r$; one corresponding to a general illumination $\vect \wf_q(\omega)$ and the other to an illumination  $\vect { e}_i$ 
whose entries are all zero except the $i$th entry which is one.  
Using the polarization identity (\ref{eq:polariden}) we can obtain $\overline{b^{(r)}_{q}}b^{(r)}_{{ e}_i}$ from 
linear combinations of the magnitudes squared $|b^{(r)}_{q}|^2$, $|b^{(r)}_{{ e}_i}|^2$, $|b^{(r)}_{q}+b^{(r)}_{{ e}_i}|^2$, 
and $|b^{(r)}_{q}+ib^{(r)}_{{ e}_i}|^2$ \cite{Moscoso16}. A physical interpretation of (\ref{familyInt0}) is as follows. Send an illumination $\vect \wf_q(\omega)$, collect the response at $\vx_r$, time reverse the received signal at $\vx_r$, and send it back to probe the medium again. Then,  $\bfm^{(r)}_{q}$ represents the signals recorded at all receivers $\vx_i$, $i=1, \dots, N$.

To wrap up, if the phases are not measured but we control the illuminations, the images can be formed  by solving (\ref{familyInt0})
using a MUSIC algorithm with several vectors $\bfm^{(r)}_{q}$  obtained in the data acquisition process. 
In the approach explained here the receiver is fixed. 
In the next subsection we explain how  to image with the MUSIC algorithm using intensity data gathered at several 
receivers.

\subsection{Several receivers}

In \cite{Novikov14}, we propose to image using MUSIC with the frequency  interferometric  matrix $\Mm(\omega)=\Pm^*(\omega)\Pm(\omega)$ which can be obtained from intensity-only measurements if the illuminations are controlled. The columns of this matrix are the vectors $\bfm^{(r)}_{q}$, $r=1,\dots,N$, obtained with the illuminations $\vect \wf_q=\vect { e}_i$, $i=1,\dots,N$.
Observe that each entry of the interferometric  matrix $\Mm(\omega)$ can be written as
\[
m_{ij} =\sum_{k=1}^N b_{k i} \bar{b}_{k j},
\]
where $b_{k i}=|b_{k i}|e^{i\theta_{ki}}$ denotes the signal (with phase) received at $\vx_k$ for illumination $\vect { e}_i$. To recover $b_{k i} \bar{b}_{k j}$ it suffices to measure the amplitudes $|b_{k i}|$, $|b_{k j}|$ and to 
find the {\it phase differences} $\theta_{ki} - \theta_{kj}$, $k=1,\dots, N$. 
The amplitudes (squared) are recorded using the  illumination vectors $\vect\we_{i}$, $i=1,2,\dots,N$. The phase differences can be recovered as follows. Since  
\[
  \theta_{ki} - \theta_{kj}=  (\theta_{k1}-\theta_{kj}) - (\theta_{k1} - \theta_{ki}),
\]
it suffices to find the  phase differences $\theta_{k1}-\theta_{kj}$ for $j=2,\dots, N$, which means that only the phase differences between 
the first vector $\vect b_1$ and all the other vectors are needed. If all $b_{k 1} \neq 0$, these
phase differences can be found from 
the polarization identities (\ref{eq:polariden}).
When the image is sparse, the assumption $b_{k 1} \neq 0$ is not restrictive because of the uncertainty principle~\cite{DS}.

Since matrices $\Mm(\omega)$  and $\Pm(\omega)$ 
have the same column space MUSIC can form the images using the SVD of $\Mm(\omega)$ and the column vectors of matrix (\ref{eq:matrix1freqbis}) as imaging vectors.

 \end{document}